\documentclass[11pt]{article}
\usepackage{amsmath, amsfonts, amssymb, amscd, amsthm, bbm}
\usepackage{graphicx}
\usepackage{caption, subcaption}

\usepackage[all]{xy}

\newtheorem{thm}{Theorem}
\newtheorem{prop}[thm]{Proposition}
\newtheorem{lem}[thm]{Lemma}

\newtheorem{cor}[thm]{Corollary}
\newtheorem{rem}[thm]{Remark}
\newtheorem{df}[thm]{Definition}

\renewcommand{\epsilon}{\varepsilon}
\renewcommand{\phi}{\varphi}

\renewcommand{\deg}{\operatorname{deg}}

\newcommand{\BB}{\mathbb}

\newcommand{\tr}{\operatorname{Tr}}

\newcommand{\HC}{\BB H_{\BB C}}

\newcommand{\ddd}{\operatorname{D}_n}

\usepackage[OT2,T1]{fontenc}

\begin{document}

\title{\bf $n$-Regular Functions in Quaternionic Analysis}
\author{Igor Frenkel and Matvei Libine}
\maketitle

\begin{abstract}
In this paper we study left and right $n$-regular functions that originally
were introduced in \cite{FL4}.
When $n=1$, these functions are the usual quaternionic left and right regular
functions.
We show that $n$-regular functions satisfy most of the properties of
the usual regular functions, including
the conformal invariance under the fractional linear transformations
by the conformal group and the Cauchy-Fueter type reproducing formulas.
Arguably, these Cauchy-Fueter type reproducing formulas for $n$-regular
functions are quaternionic analogues of Cauchy's integral formula for the
$n$-th order pole
$$
f^{(n-1)}(w) = \frac{(n-1)!}{2\pi i} \oint \frac {f(z)\,dz}{(z-w)^n}.
$$
We also find two expansions of the Cauchy-Fueter kernel for $n$-regular
functions in terms of certain basis functions,
we give an analogue of Laurent series expansion for $n$-regular functions,
we construct an invariant pairing between left and right $n$-regular functions
and we describe the irreducible representations associated to the spaces of
left and right $n$-regular functions of the conformal group and its Lie algebra.

\end{abstract}

\section{Introduction}

The foundational result in quaternionic analysis is the Cauchy-Fueter integral
formulas for left and right regular functions. Thus it is natural to ask about
quaternionic analogue of Cauchy's integral formula for the $n$-th order pole
for all positive integers $n$
\begin{equation}  \label{Cauchy}
f^{(n-1)}(w) = \frac{(n-1)!}{2\pi i} \oint \frac {f(z)\,dz}{(z-w)^n}.
\end{equation}
For $n=2$ we suggested an answer to this question in \cite{FL1} with the
first derivative replaced by the Maxwell equations in the quaternionic case.
In a recent paper \cite{FL4} we proposed a different quaternionic counterpart
of (\ref{Cauchy}), for general $n$, introducing left and right $n$-regular
functions.
For $n=1$ we get the usual left and right regular functions;
when $n=2$ we call them doubly left and right regular functions,
and in the doubly regular case the first derivative in (\ref{Cauchy})
is replaced by the degree operator plus two.

In this paper we study left and right $n$-regular functions in more detail.
Let $n$ be a positive integer and $V_{\frac n2}$ the irreducible representation
of $SU(2)$ of dimension $n+1$.
Then $n$-regular functions are functions on the space of quaternions $\BB H$
or $\BB H^{\times} = \BB H \setminus \{0\}$ with values in $V_{\frac n2}$
and satisfying $n$ regularity conditions.
The spaces of left and right $n$-regular functions form the most degenerate
series of unitary representations of $SU(2,2)$ that are often called the
spin $\frac{n}2$ representations of positive and negative helicities and
play important role in physics.
In the context of quaternionic analysis, $n$-regular functions first appeared
briefly in \cite{FL4}.
In the present paper we study these spaces in more detail as natural
generalizations of quaternionic regular functions.
The spaces of $n$-regular functions provide a class of irreducible
representations of the conformal group that were considered before,
for example, by H.~P.~Jakobsen and M.~Vergne in \cite{JV1} and in a more
general case by S.~T.~Lee \cite{Le}.

We show that $n$-regular functions satisfy most of the properties of
the usual regular functions and doubly regular functions, including
the conformal invariance under the fractional linear transformations
by the conformal group $SL(2,\HC) \simeq SL(4,\BB C)$ and the
Cauchy-Fueter type reproducing formulas.
Arguably, these Cauchy-Fueter type reproducing formulas for $n$-regular
functions are quaternionic analogues of Cauchy's integral formula for the
$n$-th order pole (\ref{Cauchy}).
%$$
%f^{(n-1)}(w) = \frac{(n-1)!}{2\pi i} \oint \frac {f(z)\,dz}{(z-w)^n}.
%$$
We also study in detail the $K$-type bases of the spaces of $n$-regular
functions, the duality between left and right regular functions and
the $\mathfrak{u}(2,2)$-invariant inner products on these spaces.

In \cite{FL4} we constructed an algebra of quaternionic functions using the
product of spaces of left and right regular functions.
We expect this construction to have a straightforward generalization to
$n$-regular functions,
thus yielding an explicit realization of an infinite family of certain
non-highest, non-lowest weight representations of the conformal group,
parametrized by positive integers $n$ with an intrinsic algebra structure.

The conformal groups of the quaternions, Minkowski space and split quaternions
are locally isomorphic to $SO(5,1)$, $SO(4,2)$ and $SO(3,3)$ respectively.
Thus our constructions complement a thorough study of minimal representations
of the indefinite orthogonal groups $O(p,q)$ by T.~Kobayashi and B.~{\O}rsted
\cite{KobO} (see references therein for the previous work on this subject).
Their work also uses the space of solutions of the ultrahyperbolic wave
equation on $\BB R^{p-1,q-1}$ to give concrete realizations of these minimal
representations.
Many results of quaternionic analysis extend to higher dimensions in the form
of Clifford analysis (see, for example, \cite{BDS, DSS, R}).
It is interesting to see how results of this paper are generalized to
Clifford analysis.

The paper is organized as follows. In Section \ref{def-section}
we define left and right $n$-regular functions
(Definitions \ref{nr-definition}, \ref{nr-definition-C}) and prove conformal
invariance under the fractional linear transformations by group $GL(2,\HC)$
(Theorem \ref{dr-action-thm}).
In Section \ref{Cauchy-Fueter-section} we derive the Cauchy-Fueter type
formulas for $n$-regular functions (Theorem \ref{Fueter-n-reg}).
In Section \ref{expansion-section} we find two expansions of the
Cauchy-Fueter kernel $k_{n/2}(Z-W)$ for $n$-regular functions
in terms of certain basis functions $F^{(n)}_{l,\mu,\nu}(Z)$,
$F'^{(n)}_{l,\mu,\nu}(Z)$, $G^{(n)}_{l,\mu,\nu}(Z)$ and $G'^{(n)}_{l,\mu,\nu}(Z)$
(Proposition \ref{expansion-prop}).
In Section \ref{deg-inverse-section} we prove a technical result that a certain
differential operator $\ddd$ that enters the Cauchy-Fueter formulas for
$n$-regular functions can be inverted for (left or right) $n$-regular functions
defined on all of $\BB H^{\times}$ (Proposition \ref{ddd-inverse-prop}).
We also give an analogue of Laurent series expansion for $n$-regular functions
on $\BB H^{\times}$ (Corollary \ref{Laurent-expansion}).
In Section \ref{bilinear-pairing-section} we use the technical result from
the previous section to construct a $\mathfrak{gl}(2,\HC)$-invariant pairing
(\ref{pairing1}) between left and right $n$-regular functions on
$\BB H^{\times}$.
We also prove orthogonality relations between the basis functions
$F^{(n)}_{l,\mu,\nu}(Z)$, $F'^{(n)}_{l,\mu,\nu}(Z)$, $G^{(n)}_{l,\mu,\nu}(Z)$ and
$G'^{(n)}_{l,\mu,\nu}(Z)$ (Proposition \ref{orthogonality-nreg-prop}).
In Section \ref{Harish-Chandra-section} we consider
$\mathfrak{gl}(2,\HC)$-modules
\begin{center}
\begin{tabular}{lcl}
${\cal F}_n^+ = \BB C\text{-span of } \bigl\{ F^{(n)}_{l,\mu,\nu}(Z) \bigr\}$,
& \qquad &
${\cal F}_n^- = \BB C\text{-span of } \bigl\{ F'^{(n)}_{l,\mu,\nu}(Z) \bigr\}$, \\
${\cal G}_n^+ = \BB C\text{-span of } \bigl\{ G^{(n)}_{l,\mu,\nu}(Z) \bigr\}$,
& \qquad &
${\cal G}_n^- = \BB C\text{-span of } \bigl\{ G'^{(n)}_{l,\mu,\nu}(Z) \bigr\}$
\end{tabular}
\end{center}
associated to left and right $n$-regular functions.
We prove that these $\mathfrak{gl}(2,\HC)$-modules are irreducible
(Theorem \ref{irreducible-thm}) and
identify their $K$-types, where $K$ is the maximal compact subgroup
$U(2) \times U(2)$ of $U(2,2)$ (Proposition \ref{K-types-prop}).
In Section \ref{unitary-section} we give explicit descriptions of the
$\mathfrak{u}(2,2)$-invariant inner products on the
$\mathfrak{gl}(2,\HC)$-modules ${\cal F}_n^{\pm}$ and ${\cal G}_n^{\pm}$
(Theorem \ref{unitary-thm}).

Since this paper is a continuation of \cite{FL1,FL3,FL4},
we follow the same notations and instead of introducing those
notations again we direct the reader to Section 2 of \cite{FL3}.

\section{Definitions and Conformal Invariance}  \label{def-section}

We continue to use notations established in \cite{FL1}.
In particular, $e_0$, $e_1$, $e_2$, $e_3$ denote the units of the classical
quaternions $\BB H$ corresponding to the more familiar $1$, $i$, $j$, $k$
(we reserve the symbol $i$ for $\sqrt{-1} \in \BB C$).
Thus $\BB H$ is an algebra over $\BB R$ generated by $e_0$, $e_1$, $e_2$, $e_3$,
and the multiplicative structure is determined by the rules
$$
e_0 e_i = e_i e_0 = e_i, \qquad
(e_i)^2 = e_1e_2e_3 = -e_0, \qquad
e_ie_j=-e_je_i, \qquad 1 \le i< j \le 3,
$$
and the fact that $\BB H$ is a division ring.
Next we consider the algebra of complexified quaternions
(also known as biquaternions) $\HC = \BB C \otimes_{\BB R} \BB H$ and
write elements of $\HC$ as
$$
Z = z^0e_0 + z^1e_1 + z^2e_2 + z^3e_3, \qquad z^0,z^1,z^2,z^3 \in \BB C,
$$
so that $Z \in \BB H$ if and only if $z^0,z^1,z^2,z^3 \in \BB R$:
$$
\BB H = \{ X = x^0e_0 + x^1e_1 + x^2e_2 + x^3e_3; \: x^0,x^1,x^2,x^3 \in \BB R \}.
$$
For $Z = z^0e_0 + z^1e_1 + z^2e_2 + z^3e_3 \in \HC$, we use
$$
Z^+ = z^0e_0 - z^1e_1 - z^2e_2 - z^3e_3
$$
to denote the quaternionic conjugation and
$$
N(Z) = Z^+Z = ZZ^+ = (z^0)^2 + (z^1)^2 + (z^2)^2 + (z^3)^2 \in \BB C
$$
to denote the quadratic norm.
Let
$$
\BB H^{\times} = \BB H \setminus \{0\} \qquad \text{and} \qquad
\HC^{\times} = \{ Z \in \HC;\: N(Z) \ne 0 \}
$$
be the sets of invertible elements in $\BB H$ and $\HC$ respectively.
The algebra $\HC$ can be naturally identified with the algebra of
$2 \times 2$ matrices with complex entries.
Recall that we denote by $\BB S$ (respectively $\BB S'$)
the irreducible 2-dimensional left (respectively right) $\HC$-module,
as described in Subsection 2.3 of \cite{FL1}.
The spaces $\BB S$ and $\BB S'$ can be realized as respectively
columns and rows of complex numbers.
Then
\begin{equation}  \label{SotimesS}
\BB S \otimes \BB S' \simeq \HC.
\end{equation}
Note that $\BB S \otimes \BB S$ and $\BB S' \otimes \BB S'$
are respectively left and right modules over $\HC \otimes \HC$.

Fix a positive integer $n$. The $n$-fold tensor product
$$
\underbrace{\BB S \otimes \cdots \otimes \BB S}_{\text{$n$ times}}
$$
is a left module over
\begin{equation}  \label{HC-n-tensor}
\underbrace{\HC \otimes \cdots \otimes \HC}_{\text{$n$ times}}
\end{equation}
and contains the $n$-fold symmetric product which we denote
$$
\underbrace{\BB S \odot \cdots \odot \BB S}_{\text{$n$ times}}
$$
as a subspace. Similarly, the $n$-fold tensor product
$$
\underbrace{\BB S' \otimes \cdots \otimes \BB S'}_{\text{$n$ times}}
$$
is a right module over (\ref{HC-n-tensor}) and contains the subspace of
$n$-fold symmetric tensors
$$
\underbrace{\BB S' \odot \cdots \odot \BB S'}_{\text{$n$ times}}.
$$
We have a natural bilinear pairing between $\BB S'$ and $\BB S$:
$$
\BB S' \times \BB S \to \BB C, \qquad
(s'_1,s'_2) \times \begin{pmatrix} s_1 \\ s_2 \end{pmatrix} \mapsto
(s'_1,s'_2) \begin{pmatrix} s_1 \\ s_2 \end{pmatrix} = s'_1s_1 + s'_2s_2.
$$
This pairing extends to a multilinear pairing
$$
\Bigl( \underbrace{\BB S' \times \cdots \times \BB S'}_{\text{$n$ times}} \Bigr)
\times
\Bigl( \underbrace{\BB S \times \cdots \times \BB S}_{\text{$n$ times}} \Bigr)
\to \BB C
$$
by taking the product of pairings of the respective components, and then to
\begin{equation}  \label{S-pairing}
\Bigl( \underbrace{\BB S' \otimes \cdots \otimes \BB S'}_{\text{$n$ times}} \Bigr)
\times
\Bigl( \underbrace{\BB S \otimes \cdots \otimes \BB S}_{\text{$n$ times}} \Bigr)
\to \BB C
\end{equation}
by multilinearity.
For convenience, we restate Lemma 5 from \cite{FL4}:

\begin{lem}  \label{antisymmetric}
Let $t \in \underbrace{\BB S \odot \cdots \odot \BB S}_{\text{$n$ times}}$ and
$t' \in \underbrace{\BB S' \odot \cdots \odot \BB S'}_{\text{$n$ times}}$,
then, for any $1 \le j,k \le n$, $j \ne k$,
$$
\sum_{i=0}^3 (1 \otimes \cdots
\underset{\text{$j$-th place}}{\otimes e_i \otimes} \cdots 
\underset{\text{$k$-th place}}{\otimes e_i \otimes} \cdots \otimes 1)t =0 \quad
\text{in $\underbrace{\BB S \otimes \cdots \otimes \BB S}_{\text{$n$ times}}$}
$$
and
$$
\sum_{i=0}^3 t'(1 \otimes \cdots
\underset{\text{$j$-th place}}{\otimes e_i \otimes} \cdots 
\underset{\text{$k$-th place}}{\otimes e_i \otimes} \cdots \otimes 1) =0 \quad
\text{in $\underbrace{\BB S' \otimes \cdots \otimes \BB S'}_{\text{$n$ times}}$}.
$$
\end{lem}

We consider spaces of functions
$$
\hat {\cal F}_n = \bigl\{ f: \HC \to
\underbrace{\BB S \otimes \cdots \otimes \BB S}_{\text{$n$ times}} \bigr\}
\qquad \text{and} \qquad
\hat {\cal G}_n = \bigl\{ g: \HC \to
\underbrace{\BB S' \otimes \cdots \otimes \BB S'}_{\text{$n$ times}} \bigr\}
$$
(possibly with singularities), and let the group $GL(2,\HC)$ act on these
spaces as follows:
\begin{equation}  \label{pi_nl}
\pi_{nl}(h): \: f(Z) \: \mapsto \: \bigl( \pi_{nl}(h)f \bigr)(Z)
= \frac{(cZ+d)^{-1} \otimes \cdots \otimes (cZ+d)^{-1}}{N(cZ+d)}
\cdot f\bigl( (aZ+b)(cZ+d)^{-1} \bigr),  \\
\end{equation}
\begin{multline}  \label{pi_nr}
\pi_{nr}(h): \: g(Z) \: \mapsto \: \bigl( \pi_{nr}(h)g \bigr)(Z) \\
= g \bigl( (a'-Zc')^{-1}(-b'+Zd') \bigr)
\cdot \frac{(a'-Zc')^{-1} \otimes \cdots \otimes (a'-Zc')^{-1}}{N(a'-Zc')},
\end{multline}
where
$f \in \hat {\cal F}_n$, $g \in \hat {\cal G}_n$,
$h = \bigl( \begin{smallmatrix} a' & b' \\ c' & d' \end{smallmatrix} \bigr)
\in GL(2,\HC)$ with
$h^{-1} = \bigl( \begin{smallmatrix} a & b \\ c & d \end{smallmatrix} \bigr)$.
Clearly, these two actions preserve the subspaces of functions with values
in $n$-fold symmetric products $\BB S \odot \cdots \odot \BB S$ and
$\BB S' \odot \cdots \odot \BB S'$ respectively.

Differentiating $\pi_{nl}$ and $\pi_{nr}$, we obtain actions of the Lie algebra
$\mathfrak{gl}(2,\HC)$, which we still denote by $\pi_{nl}$ and $\pi_{nr}$
respectively. Using notations
$$
\partial = \begin{pmatrix} \partial_{11} & \partial_{21} \\
\partial_{12} & \partial_{22} \end{pmatrix} = \frac 12 \nabla, \qquad
\partial^+ = \begin{pmatrix} \partial_{22} & -\partial_{21} \\
-\partial_{12} & \partial_{11} \end{pmatrix} = \frac 12 \nabla^+,
\qquad \partial_{ij} = \frac{\partial}{\partial z_{ij}},
$$
we can describe these actions of the Lie algebra (cf. Lemma 4 in \cite{FL4}).

\begin{lem}  \label{Lie-alg-action}
The Lie algebra action $\pi_{nl}$ of $\mathfrak{gl}(2,\HC)$ on
$\hat {\cal F}_n$ is given by
\begin{align*}
\pi_{nl} \bigl( \begin{smallmatrix} A & 0 \\ 0 & 0 \end{smallmatrix} \bigr) &:
f(Z) \mapsto - \tr (AZ \partial) f,  \\
\pi_{nl} \bigl( \begin{smallmatrix} 0 & B \\ 0 & 0 \end{smallmatrix} \bigr) &:
f(Z) \mapsto - \tr (B \partial) f,  \\
\pi_{nl} \bigl( \begin{smallmatrix} 0 & 0 \\ C & 0 \end{smallmatrix} \bigr) &:
f(Z) \mapsto \tr (ZCZ \partial +CZ) f \\
& \phantom{: f(Z) \mapsto \tr (}
+ (CZ \otimes 1 \otimes \cdots \otimes 1 + 1 \otimes CZ \otimes \cdots \otimes 1
+ \cdots + 1 \otimes \cdots \otimes 1 \otimes CZ) f,  \\
\pi_{nl} \bigl( \begin{smallmatrix} 0 & 0 \\ 0 & D \end{smallmatrix} \bigr) &:
f(Z) \mapsto \tr (ZD \partial +D) f \\
& \phantom{: f(Z) \mapsto \tr (}
+ (D \otimes 1 \otimes \cdots \otimes 1 + 1 \otimes D \otimes \cdots \otimes 1
+ \cdots + 1 \otimes \cdots \otimes 1 \otimes D) f.
\end{align*}

Similarly, the Lie algebra action $\pi_{nr}$ of $\mathfrak{gl}(2,\HC)$ on
$\hat {\cal G}_n$ is given by
\begin{align*}
\pi_{nr} \bigl( \begin{smallmatrix} A & 0 \\ 0 & 0 \end{smallmatrix} \bigr) &:
g(Z) \mapsto - \tr (AZ \partial +A) g \\
& \phantom{: g(Z) \mapsto - \tr (}
-g (A \otimes 1 \otimes \cdots \otimes 1 + 1 \otimes A \otimes \cdots \otimes 1
+ \cdots + 1 \otimes \cdots \otimes 1 \otimes A),  \\
\pi_{nr} \bigl( \begin{smallmatrix} 0 & B \\ 0 & 0 \end{smallmatrix} \bigr) &:
g(Z) \mapsto - \tr (B \partial) g,  \\
\pi_{nr} \bigl( \begin{smallmatrix} 0 & 0 \\ C & 0 \end{smallmatrix} \bigr) &:
g(Z) \mapsto \tr (ZCZ \partial +ZC) g \\
& \phantom{: g(Z) \mapsto \tr (}
+g (ZC \otimes 1 \otimes \cdots \otimes 1 + 1 \otimes ZC \otimes \cdots \otimes 1
+ \cdots + 1 \otimes \cdots \otimes 1 \otimes ZC),  \\
\pi_{nr} \bigl( \begin{smallmatrix} 0 & 0 \\ 0 & D \end{smallmatrix} \bigr) &:
g(Z) \mapsto \tr (ZD \partial) g.
\end{align*}
\end{lem}

\begin{proof}
These formulas are obtained by differentiating (\ref{pi_nl}) and (\ref{pi_nr}).
\end{proof}

We introduce $2n$ first order differential operators
\begin{multline*}
\underset{k}{\nabla^+} = (1 \otimes \cdots
\underset{\text{$k$-th place}}{\otimes e_0 \otimes} \cdots \otimes 1)
\frac{\partial}{\partial x^0}
+ (1 \otimes \cdots \underset{\text{$k$-th place}}{\otimes e_1 \otimes}
\cdots \otimes 1) \frac{\partial}{\partial x^1}  \\
+ (1 \otimes \cdots \underset{\text{$k$-th place}}{\otimes e_2 \otimes}
\cdots \otimes 1) \frac{\partial}{\partial x^2}
+ (1 \otimes \cdots \underset{\text{$k$-th place}}{\otimes e_3 \otimes}
\cdots \otimes 1) \frac{\partial}{\partial x^3},
\end{multline*}
\begin{multline*}
\underset{k}{\nabla} = \frac{\partial}{\partial x^0} (1 \otimes \cdots
\underset{\text{$k$-th place}}{\otimes e_0 \otimes} \cdots \otimes 1)
- \frac{\partial}{\partial x^1} (1 \otimes \cdots
\underset{\text{$k$-th place}}{\otimes e_1 \otimes} \cdots \otimes 1)  \\
- \frac{\partial}{\partial x^2}(1 \otimes \cdots
\underset{\text{$k$-th place}}{\otimes e_2 \otimes} \cdots \otimes 1)
- \frac{\partial}{\partial x^3}(1 \otimes \cdots
\underset{\text{$k$-th place}}{\otimes e_3 \otimes} \cdots \otimes 1),
\end{multline*}
$k=1,\dots,n$, which can be applied to functions with values in $n$-fold
tensor products $\BB S \otimes \cdots \otimes \BB S$ or
$\BB S' \otimes \cdots \otimes \BB S'$ as follows.
If $U$ is an open subset of $\BB H$ or $\HC$ and
$f: U \to \BB S \otimes \cdots \otimes \BB S$ is a differentiable function,
then these operators can be applied to $f$ on the left. For example,
\begin{multline*}
\underset{1}{\nabla^+} f = (e_0 \otimes 1 \otimes \cdots \otimes 1)
\frac{\partial f}{\partial x^0}
+ (e_1 \otimes 1 \otimes \cdots \otimes 1) \frac{\partial f}{\partial x^1}  \\
+ (e_2 \otimes 1 \otimes \cdots \otimes 1) \frac{\partial f}{\partial x^2}
+ (e_3 \otimes 1 \otimes \cdots \otimes 1) \frac{\partial f}{\partial x^3}.
\end{multline*}
Similarly, these operators can be applied on the right to
differentiable functions $g: U \to \BB S' \otimes \cdots \otimes \BB S'$;
we often indicate this with an arrow above the operator. For example,
\begin{multline*}
g \overleftarrow{\underset{n}{\nabla^+}}
= \frac{\partial g}{\partial x^0} (1 \otimes \cdots \otimes 1 \otimes e_0)
+ \frac{\partial g}{\partial x^1} (1 \otimes \cdots \otimes 1 \otimes e_1)  \\
+ \frac{\partial g}{\partial x^2} (1 \otimes \cdots \otimes 1 \otimes e_2)
+ \frac{\partial g}{\partial x^3} (1 \otimes \cdots \otimes 1 \otimes e_3).
\end{multline*}

\begin{df}  \label{nr-definition}
Let $U$ be an open subset of $\BB H$.
A ${\cal C}^1$-function $f: U \to \BB S \odot \cdots \odot \BB S$ is
{\em left $n$-regular} if it satisfies $n$ differential equations
$$
\underset{k}{\nabla^+} f =0, \qquad k=1,\dots,n,
$$
for all points in $U$.
Similarly, a ${\cal C}^1$-function $g: U \to \BB S' \odot \cdots \odot \BB S'$
is {\em right $n$-regular} if
$$
g \overleftarrow{\underset{k}{\nabla^+}} =0, \qquad k=1,\dots,n,
$$
for all points in $U$.
\end{df}

Since
$$
\underset{k}{\nabla^+} \underset{k}{\nabla}
= \underset{k}{\nabla} \underset{k}{\nabla^+}
=\square, \qquad k=1,\dots,n,
$$
where
$$
\square = \frac{\partial^2}{(\partial x^0)^2}+
\frac{\partial^2}{(\partial x^1)^2} + \frac{\partial^2}{(\partial x^2)^2}+
\frac{\partial^2}{(\partial x^3)^2},
$$
left and right $n$-regular functions are harmonic.

One way to construct $n$-regular functions is to start with a harmonic
function $\phi: \BB H \to \BB S \odot \cdots \odot \BB S$, then
$(\nabla \otimes \cdots \otimes \nabla) \phi$ is left $n$-regular.
Similarly, if $\phi: \BB H \to \BB S' \odot \cdots \odot \BB S'$ is harmonic,
then $\phi (\overleftarrow{\nabla \otimes \cdots \otimes \nabla})$
is right $n$-regular.

We also can talk about $n$-regular functions defined on open subsets of
$\HC$. In this case we require such functions to be holomorphic.

\begin{df}  \label{nr-definition-C}
Let $U$ be an open subset of $\HC$.
A holomorphic function $f: U \to \BB S \odot \cdots \odot \BB S$ is
{\em left $n$-regular} if it satisfies $n$ differential equations
$$
\underset{k}{\nabla^+} f =0, \qquad k=1,\dots,n,
$$
for all points in $U$.

Similarly, a holomorphic function $g: U \to \BB S' \odot \cdots \odot \BB S'$
is {\em right $n$-regular} if
$$
g \overleftarrow{\underset{k}{\nabla^+}} =0, \qquad k=1,\dots,n,
$$
for all points in $U$.
\end{df}

Let ${\cal R}_n$ and ${\cal R}'_n$ denote respectively the spaces of
(holomorphic) left and right $n$-regular functions on $\HC$,
possibly with singularities.

\begin{thm}  \label{dr-action-thm}
\begin{enumerate}
\item
The space ${\cal R}_n$ of left $n$-regular functions
$\HC \to \BB S \odot \cdots \odot \BB S$ (possibly with singularities)
is invariant under the $\pi_{nl}$ action (\ref{pi_nl}) of $GL(2,\HC)$.
\item
The space ${\cal R}'_n$ of right $n$-regular functions
$\HC \to \BB S' \odot \cdots \odot \BB S'$ (possibly with singularities)
is invariant under the $\pi_{nr}$ action (\ref{pi_nr}) of $GL(2,\HC)$.
\end{enumerate}
\end{thm}

\begin{proof}
Since the Lie group $GL(2,\HC) \simeq GL(4,\BB C)$ is connected, it is
sufficient to show that, if $f \in {\cal R}_n$, $g \in {\cal R}'_n$ and
$\bigl( \begin{smallmatrix} A & B \\ C & D \end{smallmatrix} \bigr) \in
\mathfrak{gl}(2,\HC)$, then
$\pi_{nl} \bigl( \begin{smallmatrix} A & B \\ C & D \end{smallmatrix} \bigr)f
\in {\cal R}_n$ and
$\pi_{nr} \bigl( \begin{smallmatrix} A & B \\ C & D \end{smallmatrix} \bigr)g
\in {\cal R}'_n$.
Consider, for example, the case of 
$\pi_{nl} \bigl( \begin{smallmatrix} 0 & 0 \\ C & 0 \end{smallmatrix} \bigr)f$,
the other cases are similar. For $k=1,\dots,n$, we have:
\begin{multline*}
\underset{k}{\nabla^+} \pi_{nl}
\bigl( \begin{smallmatrix} 0 & 0 \\ C & 0 \end{smallmatrix} \bigr)f
= \underset{k}{\nabla^+} \Bigl( \tr(ZCZ \partial +CZ)f
+ \sum_{j=1}^n (1 \otimes \cdots
\underset{\text{$j$-th place}}{\otimes CZ \otimes} \cdots \otimes 1)f \Bigr)  \\
= \underset{k}{\nabla^+} \Bigl( \tr(ZCZ \partial +CZ)f +
(1 \otimes \cdots
\underset{\text{$k$-th place}}{\otimes CZ \otimes} \cdots \otimes 1)f \Bigr)  \\
+ \sum_{j \ne k} (1 \otimes \cdots
\underset{\text{$j$-th place}}{\otimes CZ \otimes} \cdots \otimes 1)
\underset{k}{\nabla^+}f \\
+ \sum_{j \ne k} (1 \otimes \cdots
\underset{\text{$j$-th place}}{\otimes C \otimes} \cdots \otimes 1)
\sum_{i=0}^3 (1 \otimes \cdots
\underset{\text{$j$-th place}}{\otimes e_i \otimes} \cdots 
\underset{\text{$k$-th place}}{\otimes e_i \otimes} \cdots \otimes 1)f,
\end{multline*}
the first summand is zero essentially because the space of left regular
functions is invariant under the action $\pi_l$ (equation (22) in \cite{FL1}),
the second summand is zero because $f$ satisfies $\underset{k}{\nabla^+} f=0$,
and the third summand is zero by Lemma \ref{antisymmetric}.
\end{proof}

\section{Cauchy-Fueter Formulas for $n$-Regular Functions}  \label{Cauchy-Fueter-section}

In this section we derive the Cauchy-Fueter type formulas for $n$-regular
functions from the classical Cauchy-Fueter formulas for left and right
regular functions.

\begin{lem}  \label{zf-regular}
Let $f(Z)$ be a left $n$-regular function, then the
$\BB S \otimes \cdots \otimes \BB S$-valued functions
$$
(Z \otimes \cdots \underset{\text{$k$-th place}}{\otimes 1 \otimes}
\cdots \otimes Z) f(Z), \qquad k=1,\dots,n,
$$
are ``left regular in $k$-th place'' in the sense that they satisfy
$$
\underset{k}{\nabla^+} \Bigl[
(Z \otimes \cdots \underset{\text{$k$-th place}}{\otimes 1 \otimes}
\cdots \otimes Z) f(Z) \Bigr] = 0.
$$

Similarly, if $g(Z)$ is a right $n$-regular function, then the
$\BB S' \otimes \cdots \otimes \BB S'$-valued functions
$$
g(Z) (Z \otimes \cdots \underset{\text{$k$-th place}}{\otimes 1 \otimes}
\cdots \otimes Z), \qquad k=1,\dots,n,
$$
are ``right regular in $k$-th place'' in the sense that they satisfy
$$
\Bigl[ g(Z) (Z \otimes \cdots \underset{\text{$k$-th place}}{\otimes 1 \otimes}
\cdots \otimes Z) \Bigr] \overleftarrow{\underset{k}{\nabla^+}} =0.
$$
\end{lem}

\begin{proof}
We have:
\begin{multline*}
\underset{k}{\nabla^+} \Bigl[
(Z \otimes \cdots \underset{\text{$k$-th place}}{\otimes 1 \otimes}
\cdots \otimes Z) f(Z) \Bigr]
= (Z \otimes \cdots \underset{\text{$k$-th place}}{\otimes 1 \otimes}
\cdots \otimes Z) \Bigl[ \underset{k}{\nabla^+} f(Z) \Bigr]  \\
+ \sum_{j \ne k} \sum_{i=0}^3 (Z \otimes \cdots
\underset{\text{$j$-th place}}{\otimes e_i \otimes} \cdots 
\underset{\text{$k$-th place}}{\otimes e_i \otimes} \cdots \otimes Z) f(Z),
\end{multline*}
the first term is zero because $f$ satisfies $\underset{k}{\nabla^+} f=0$
and the second term is zero by Lemma \ref{antisymmetric}.
Proof of the second part of the lemma is similar.
\end{proof}

Recall the quaternionic valued holomorphic $3$-form $Dz$ on $\HC$:
$$
Dz = e_0 dz^1 \wedge dz^2 \wedge dz^3 - e_1 dz^0 \wedge dz^2 \wedge dz^3
+ e_2 dz^0 \wedge dz^1 \wedge dz^3 - e_3 dz^0 \wedge dz^1 \wedge dz^2.
$$
We also consider holomorphic $3$-forms
$$
Z \otimes \cdots \underset{\text{$k$-th place}}{\otimes Dz \otimes}
\cdots \otimes Z, \qquad k=1,\dots,n,
$$
on $\HC$ with values in $\HC \otimes \cdots \otimes \HC$.
We have an analogue of Cauchy's integral theorem for $n$-regular functions.

\begin{lem}  \label{Cauchy-thm}
Let $U \subset \BB H$ be an open bounded subset with piecewise ${\cal C}^1$
boundary $\partial U$. Suppose that $f(Z)$ is left $n$-regular and
$g(Z)$ is right $n$-regular on a neighborhood of the closure $\overline{U}$.
Then, for each $k=1,\dots,n$,
$$
\int_{\partial U} g(Z) \cdot (Z \otimes \cdots
\underset{\text{$k$-th place}}{\otimes Dz \otimes} \cdots \otimes Z)
\cdot f(Z) =0.
$$
\end{lem}

Note that the expression inside the integral is a $\BB C$-valued function
obtained by applying the pairing (\ref{S-pairing}).

\begin{proof}
Essentially by the definition of $Dz$,
\begin{multline*}
d \Bigl[ g(Z) \cdot (Z \otimes \cdots
\underset{\text{$k$-th place}}{\otimes Dz \otimes} \cdots \otimes Z)
\cdot f(Z) \Bigr]
= \Bigl[ g(Z) \overleftarrow{\underset{k}{\nabla^+}} \cdot (Z \otimes \cdots
\underset{\text{$k$-th place}}{\otimes 1 \otimes} \cdots \otimes Z)
\cdot f(Z) \\
+ g(Z) \cdot \underset{k}{\nabla^+} \Bigl( (Z \otimes \cdots
\underset{\text{$k$-th place}}{\otimes 1 \otimes} \cdots \otimes Z) \cdot f(Z)
\Bigr) \Bigr] dz^0 \wedge dz^1 \wedge dz^2 \wedge dz^3.
\end{multline*}
By Lemma \ref{zf-regular}, both summands are zero, and the result follows.
\end{proof}

Recall the degree operator $\deg$ acting on functions on $\BB H$:
$$
\deg f = x^0\frac{\partial f}{\partial x^0} +
x^1\frac{\partial f}{\partial x^1} + x^2\frac{\partial f}{\partial x^2}
+ x^3\frac{\partial f}{\partial x^3},
$$
and let $(\deg+m)$, $m \in \BB Z$,  denote the degree operator plus
$m$ times the identity:
$$
(\deg+m) f = \deg f + mf.
$$
Similarly, we can define operators $\deg$ and $(\deg+m)$ acting on functions
on $\HC$. For convenience we recall Lemma 8 from \cite{FL2}
(it applies to both cases).

\begin{lem}
\begin{equation}  \label{deg-nabla}
2(\deg+2) = Z^+\nabla^+ + \nabla Z = \nabla^+Z^+ + Z\nabla.
\end{equation}
\end{lem}

We introduce an operator
$$
\ddd = \underbrace{(\deg+n)(\deg+n-1) \cdots (\deg+2)}_{\text{$n-1$ operators}}.
$$

Define a function of $Z$ and $W$ taking values in
$\HC \otimes \cdots \otimes \HC$
\begin{equation}  \label{k_n-kernel}
k_{n/2}(Z-W) = \frac1{2^n} (\underbrace{\nabla_W \otimes \cdots \otimes \nabla_W}
_{\text{$n$ times}}) \frac1{N(Z-W)}
= \frac{(-1)^n}{2^n} (\underbrace{\nabla_Z \otimes \cdots \otimes \nabla_Z}
_{\text{$n$ times}}) \frac1{N(Z-W)}.
\end{equation}
Observe that
\begin{equation}  \label{k_n-symmetry}
k_{n/2}(Z-W) = (-1)^n k_{n/2}(W-Z).
\end{equation}
We have the following analogue of the Cauchy-Fueter formulas for
$n$-regular functions.

\begin{thm}  \label{Fueter-n-reg}
Let $U \subset \BB H$ be an open bounded subset with piecewise ${\cal C}^1$
boundary $\partial U$. Suppose that $f(Z)$ is left $n$-regular on a
neighborhood of the closure $\overline{U}$, then, for each $k=1,\dots,n$,
\begin{equation*}
\frac1{2\pi^2} \int_{\partial U} k_{n/2}(Z-W) \cdot
(Z \otimes \cdots \underset{\text{$k$-th place}}{\otimes Dz \otimes}
\cdots \otimes Z) \cdot f(Z)
= \begin{cases} \ddd f(W) & \text{if $W \in U$;} \\
0 & \text{if $W \notin \overline{U}$.}
\end{cases}
\end{equation*}
If $g(Z)$ is right $n$-regular on a neighborhood of the closure
$\overline{U}$, then, for each $k=1,\dots,n$,
\begin{equation*}
\frac1{2\pi^2} \int_{\partial U} g(Z) \cdot (Z \otimes \cdots
\underset{\text{$k$-th place}}{\otimes Dz \otimes} \cdots \otimes Z)
\cdot k_{n/2}(Z-W)
= \begin{cases}
\ddd  g(W) & \text{if $W \in U$;} \\
0 & \text{if $W \notin \overline{U}$.}
\end{cases}
\end{equation*}
\end{thm}

\begin{proof}
By Lemma \ref{zf-regular}, the $\BB S \otimes \cdots \otimes \BB S$-valued
function
$$
(Z \otimes \cdots
\underset{\text{$k$-th place}}{\otimes 1 \otimes} \cdots \otimes Z) f(Z)
$$
is ``left regular in the $k$-th place'' in the sense that it is annihilated by
$\underset{k}{\nabla^+}$.
From the classical Cauchy-Fueter formula for left regular functions,
we obtain:
\begin{multline}  \label{old-Fueter}
\frac1{2\pi^2} \int_{\partial U} \tilde k_{1/2}(Z-W) \cdot (1 \otimes \cdots
\underset{\text{$k$-th place}}{\otimes Dz \otimes} \cdots \otimes 1) \cdot
(Z \otimes \cdots
\underset{\text{$k$-th place}}{\otimes 1 \otimes} \cdots \otimes Z) f(Z) \\
= \begin{cases}
(W \otimes \cdots
\underset{\text{$k$-th place}}{\otimes 1 \otimes} \cdots \otimes W) f(W) &
\text{if $W \in U$;} \\
0 & \text{if $W \notin \overline{U}$,}
\end{cases}
\end{multline}
where
$$
\tilde k_{1/2}(Z-W) = 1 \otimes \cdots \underset{\text{$k$-th place}}
{\otimes \frac{(Z-W)^{-1}}{N(Z-W)} \otimes} \cdots \otimes 1
= \frac12 \underset{k}{\nabla_W} \frac1{N(Z-W)}
= -\frac12 \underset{k}{\nabla_Z} \frac1{N(Z-W)}.
$$
Applying $n-1$ differential operators
$\underset{j}{\nabla}$, $j=1,\dots,n$, $j \ne k$,
%$\nabla \otimes \cdots \underset{\text{$k$-th place}}{\otimes 1 \otimes}
%\cdots \otimes \nabla$
to both sides of (\ref{old-Fueter})
(the derivative is taken with respect to $W$),
\begin{multline*}
\frac{2^{n-1}}{2\pi^2} \int_{\partial U} k_{n/2}(Z-W) \cdot
(Z \otimes \cdots \underset{\text{$k$-th place}}{\otimes Dz \otimes}
\cdots \otimes Z) \cdot f(Z) \\
= \begin{cases}
(\nabla \otimes \cdots \underset{\text{$k$-th place}}{\otimes 1 \otimes}
\cdots \otimes \nabla)
(W \otimes \cdots \underset{\text{$k$-th place}}{\otimes 1 \otimes}
\cdots \otimes W) f(W) & \text{if $W \in U$;} \\
0 & \text{if $W \notin \overline{U}$}
\end{cases} \\
= \begin{cases}
2^{n-1} \underbrace{(\deg+n)(\deg+n-1) \cdots (\deg+2)}_{\text{$n-1$ operators}} f(W)
& \text{if $W \in U$;} \\
0 & \text{if $W \notin \overline{U}$,}
\end{cases}
\end{multline*}
where the last equality follows from (\ref{deg-nabla}) and
Lemma \ref{antisymmetric}, since $\underset{j}{\nabla^+} f =0$,
for each $j=1,\dots,n$.
The case of right $n$-regular function is similar.
\end{proof}

We have an analogue of Liouville's theorem for $n$-regular functions:

\begin{cor}  \label{Liouville}
Let $f: \BB H \to \BB S \odot \cdots \odot \BB S$ be a function that is
left $n$-regular and bounded on $\BB H$, then $f$ is constant.
Similarly, if $g: \BB H \to \BB S' \odot \cdots \odot \BB S'$ is a function
that is right $n$-regular and bounded on $\BB H$, then $g$ is constant.
\end{cor}

\begin{proof}
The proof is essentially the same as for the (classical) left and right
regular functions on $\BB H$, so we only give a sketch of the first part.
From Theorem \ref{Fueter-n-reg} we have:
\begin{equation*}
\frac{\partial}{\partial x^0} \ddd f(X)
= \frac1{2\pi^2} \int_{S^3_R} \frac{\partial k_{n/2}(Z-X)}{\partial x^0}
\cdot (Dz \otimes Z \otimes \cdots \otimes Z) \cdot f(Z),
\end{equation*}
where $S^3_R \subset \BB H$ is the three-dimensional sphere of radius $R$
centered at the origin
$$
S^3_R = \{ X \in \BB H ;\: N(X)=R^2 \}
$$
with $R^2>N(X)$.
If $f$ is bounded, one easily shows that the integral on the right hand side
tends to zero as $R \to \infty$. Thus
$$
\frac{\partial}{\partial x^0}  \ddd f =0.
$$
Similarly, the other partial derivatives
$$
\frac{\partial}{\partial x^1}  \ddd f
= \frac{\partial}{\partial x^2}  \ddd f
= \frac{\partial}{\partial x^3}  \ddd f = 0.
$$
It follows that $\ddd f$ and hence $f$
are constant.
\end{proof}

\section{Expansion of the Cauchy-Fueter Kernel for $n$-Regular Functions}  \label{expansion-section}

We often identify $\HC$ with $2 \times 2$ matrices with complex entries.
Similarly, it will be convenient to identify
$\underbrace{\HC \otimes \cdots \otimes \HC}_{\text{$n$ times}}$
with complex $2^n \times 2^n$ matrices using the Kronecker product.
Let $\BB C^{k \times k}$ denote the algebra of $k \times k$ complex matrices.
For example, if $A = \bigl( \begin{smallmatrix} a_{11} & a_{12} \\
a_{21} & a_{22} \end{smallmatrix} \bigr),
B = \bigl( \begin{smallmatrix} b_{11} & b_{12} \\
b_{21} & b_{22} \end{smallmatrix} \bigr) \in \BB C^{2 \times 2}$,
then their Kronecker product is
$$
A \otimes B = \begin{pmatrix} a_{11}B & a_{12}B \\ a_{21}B & a_{22}B \end{pmatrix}
= \begin{pmatrix} a_{11}b_{11} & a_{11}b_{12} & a_{12}b_{11} & a_{12}b_{12} \\
a_{11}b_{21} & a_{11}b_{22} & a_{12}b_{21} & a_{12}b_{22} \\
a_{21}b_{11} & a_{21}b_{12} & a_{22}b_{11} & a_{22}b_{12} \\
a_{21}b_{21} & a_{21}b_{22} & a_{22}b_{21} & a_{22}b_{22} \end{pmatrix}
\in \BB C^{4 \times 4}.
$$
It is easy to see that the Kronecker product satisfies
$$
(A \otimes B) (C \otimes D) = (AC) \otimes (BD).
$$

Next we recall the matrix coefficients $t^l_{\nu\,\underline{\mu}}(Z)$'s of $SU(2)$
described by equation (27) of \cite{FL1} (cf. \cite{V}):
\begin{equation}  \label{t}
t^l_{\nu\,\underline{\mu}}(Z) = \frac 1{2\pi i}
\oint (sz_{11}+z_{21})^{l-\mu} (sz_{12}+z_{22})^{l+\mu} s^{-l+\nu} \,\frac{ds}s,
\qquad
\begin{smallmatrix} l = 0, \frac12, 1, \frac32, \dots, \\
\mu,\nu \in \BB Z +l, \\ -l \le \mu,\nu \le l, \end{smallmatrix}
\end{equation}
$Z=\bigl(\begin{smallmatrix} z_{11} & z_{12} \\
z_{21} & z_{22} \end{smallmatrix}\bigr) \in \HC$,
the integral is taken over a loop in $\BB C$ going once around the origin
in the counterclockwise direction.
We regard these functions as polynomials on $\HC$.

Since $\square t^l_{\nu\,\underline{\mu}}(Z) =0$ and
$\square \bigl( N(Z)^{-1} \cdot t^l_{\mu\,\underline{\nu}}(Z^{-1}) \bigr) =0$,
by the observation made after Definition \ref {nr-definition},
the columns and rows of the two $2^n \times 2^n$ matrices
$$
(\underbrace{\partial \otimes \cdots \otimes \partial}_{\text{$n$ times}})
t^l_{\nu\,\underline{\mu}}(Z)
\qquad \text{and} \qquad
(\underbrace{\partial \otimes \cdots \otimes \partial}_{\text{$n$ times}})
\bigl( N(Z)^{-1} \cdot t^l_{\nu\,\underline{\mu}}(Z^{-1}) \bigr)
$$
are respectively left and right $n$-regular.
We use this to construct bases of left and right $n$-regular functions.
Let indices $l$, $\mu$ and $\nu$ range as follows:
$$
l = 0, \frac12, 1, \frac32, \dots, \quad
\mu \in \BB Z +l+n/2, \quad
\nu \in \BB Z +l, \quad
-l-n/2 \le \mu \le l+n/2, \quad
-l \le \nu \le l.
$$
Introduce left $n$-regular functions $\HC \to \BB S \odot \cdots \odot \BB S$
\begin{equation}  \label{F-def}
F^{(n)}_{l,\mu,\nu}(Z) = \biggl( \underbrace{
\begin{pmatrix} \partial_{11} \\ \partial_{12} \end{pmatrix}
\otimes \cdots \otimes
\begin{pmatrix} \partial_{11} \\ \partial_{12} \end{pmatrix}}_{\text{$n$ times}}
\biggr) t^{l+n/2}_{\nu-n/2\,\underline{\mu}}(Z)
= \biggl( \underbrace{
\begin{pmatrix} \partial_{21} \\ \partial_{22} \end{pmatrix}
\otimes \cdots \otimes
\begin{pmatrix} \partial_{21} \\ \partial_{22} \end{pmatrix}}_{\text{$n$ times}}
\biggr) t^{l+n/2}_{\nu+n/2\,\underline{\mu}}(Z)
\end{equation}
and right $n$-regular functions $\HC \to \BB S' \odot \cdots \odot \BB S'$
$$
G^{(n)}_{l,\mu,\nu}(Z) = %\frac1{(l-\nu+n) (l-\mu+n-1) \cdots (l-\nu+1)}
\frac{(l-\nu)!}{(l-\nu+n)!}
\bigl( \underbrace{ (\partial_{11}, \partial_{21}) \otimes \cdots \otimes
(\partial_{11}, \partial_{21})}_{\text{$n$ times}} \bigr)
t^{l+n/2}_{\mu\,\underline{\nu-n/2}}(Z).
$$
In order to construct the dual basis, we consider left $n$-regular
functions $\HC^{\times} \to \BB S \odot \cdots \odot \BB S$
$$
(-1)^n \biggl( \underbrace{
\begin{pmatrix} \partial_{11} \\ \partial_{12} \end{pmatrix}
\otimes \cdots \otimes
\begin{pmatrix} \partial_{11} \\ \partial_{12} \end{pmatrix}}_{\text{$n$ times}}
 \biggr) \bigl( N(Z)^{-1} \cdot t^l_{\nu\,\underline{\tilde\mu}}(Z^{-1}) \bigr);
$$
using Lemmas 22 and 23 in \cite{FL1}, one can see that
%by equation (37) in \cite{FL4},
these are columns with entries being some scalar multiples of
$$
N(Z)^{-1} \cdot t^{l+n/2}_{\nu-n/2\,\underline{\tilde\mu-n/2}}(Z^{-1}), \dots,
N(Z)^{-1} \cdot t^{l+n/2}_{\nu+n/2\,\underline{\tilde\mu-n/2}}(Z^{-1}).
$$
Shift the index $\tilde\mu$ so that $\tilde\mu-n/2 = \mu$ and call the
resulting function $F'^{(n)}_{l,\mu,\nu}(Z)$.
Note that this is not always the same as differentiating
$N(Z)^{-1} \cdot t^l_{\nu\,\underline{\mu+n/2}}(Z^{-1})$, since
$\mu+n/2$ may be bigger than $l$, which is outside of allowed range.
%$$
%F'^{(n)}_{l,\mu,\nu}(Z) = (-1)^n \biggl( \underbrace{
%\begin{pmatrix} \partial_{11} \\ \partial_{12} \end{pmatrix}
%\otimes \cdots \otimes
%\begin{pmatrix} \partial_{11} \\ \partial_{12} \end{pmatrix}}_{\text{$n$ times}}
%\bigl( N(Z)^{-1} \cdot t^l_{\nu\,\underline{\mu+n/2}}(Z^{-1}) \bigr)
%$$
Similarly, consider right $n$-regular functions
$\HC^{\times} \to \BB S' \odot \cdots \odot \BB S'$
$$
(-1)^n \frac{(l-\tilde\mu)!}{(l-\tilde\mu+n)!}
\bigl( \underbrace{ (\partial_{11}, \partial_{21}) \otimes \cdots \otimes
(\partial_{11}, \partial_{21})}_{\text{$n$ times}} \bigr)
\bigl( N(Z)^{-1} \cdot t^l_{\tilde\mu\,\underline{\nu}}(Z^{-1}) \bigr);
$$
using Lemmas 22 and 23 in \cite{FL1}, one can see that
%by equation (37) in \cite{FL4},
these are rows with entries being
$$
N(Z)^{-1} \cdot t^{l+n/2}_{\tilde\mu-n/2\,\underline{\nu-n/2}}(Z^{-1}), \dots,
N(Z)^{-1} \cdot t^{l+n/2}_{\tilde\mu-n/2\,\underline{\nu+n/2}}(Z^{-1}).
$$
Shift the index $\tilde\mu$ so that $\tilde\mu-n/2 = \mu$ and call the
resulting function $G'^{(n)}_{l,\mu,\nu}(Z)$.
Again, this is not always the same as differentiating
$N(Z)^{-1} \cdot t^l_{\mu+n/2\,\underline{\nu}}(Z^{-1})$, since
$\mu+n/2$ may be bigger than $l$.
%$$
%G'^{(n)}_{l,m,n}(Z) = (-1)^n \frac{(l-\mu-n/2)!}{(l-\mu+n/2)!}
%\bigl( \underbrace{ (\partial_{11}, \partial_{21}) \otimes \cdots \otimes
%(\partial_{11}, \partial_{21})}_{\text{$n$ times}} \bigr)
%\bigl( N(Z)^{-1} \cdot t^l_{\mu+n/2\,\underline{\nu}}(Z^{-1}) \bigr)
%$$
Note that $F^{(1)}_{l,\mu,\nu}(Z)$, $F'^{(1)}_{l,\mu,\nu}(Z)$,
$G^{(1)}_{l,\mu,\nu}(Z)$ and $G'^{(1)}_{l,\mu,\nu}(Z)$ are exactly the basis
functions that appear in Proposition 24 in \cite{FL1}:
\begin{align*}
F^{(1)}_{l,\mu,\nu}(Z)
&= \begin{pmatrix} (l-\mu+1/2)t^l_{\nu\,\underline{\mu+1/2}}(Z) \\
(l+\mu+1/2)t^l_{\nu\,\underline{\mu-1/2}}(Z) \end{pmatrix},  \\
F'^{(1)}_{l,\mu,\nu}(Z) &= \begin{pmatrix}
(l-\nu+1) N(Z)^{-1} \cdot t^{l+1/2}_{\nu-1/2\,\underline{\mu}}(Z^{-1}) \\
(l+\nu+1) N(Z)^{-1} \cdot t^{l+1/2}_{\nu+1/2\,\underline{\mu}}(Z^{-1})
\end{pmatrix},  \\
G^{(1)}_{l,\mu,\nu}(Z) &= \bigl( t^l_{\mu+1/2\,\underline{\nu}}(Z),
t^l_{\mu-1/2\,\underline{\nu}}(Z) \bigr),  \\
G'^{(1)}_{l,\mu,\nu}(Z) &=
\bigl( N(Z)^{-1} \cdot t^{l+1/2}_{\mu\,\underline{\nu-1/2}}(Z^{-1}),
N(Z)^{-1} \cdot t^{l+1/2}_{\mu\,\underline{\nu+1/2}}(Z^{-1}) \bigr).
\end{align*}

\begin{lem}  \label{recursive-lem}
We have the following recursive relations between the functions
$F^{(n)}_{l,\mu,\nu}(Z)$, $F'^{(n)}_{l,\mu,\nu}(Z)$, $G^{(n)}_{l,\mu,\nu}(Z)$ and
$G'^{(n)}_{l,\mu,\nu}(Z)$.
$$
F^{(n+1)}_{l,\mu,\nu}(Z) =
\begin{pmatrix} \partial_{11} F^{(n)}_{l+1/2,\mu,\nu-1/2}(Z) \\
\partial_{12} F^{(n)}_{l+1/2,\mu,\nu-1/2}(Z) \end{pmatrix}
= \begin{pmatrix} \partial_{21} F^{(n)}_{l+1/2,\mu,\nu+1/2}(Z) \\
\partial_{22} F^{(n)}_{l+1/2,\mu,\nu+1/2}(Z) \end{pmatrix},
$$
$$
F'^{(n+1)}_{l,\mu-1/2,\nu}(Z) =
- \begin{pmatrix} \partial_{11} F'^{(n)}_{l,\mu,\nu}(Z) \\
\partial_{12} F^{(n)}_{l,\mu,\nu}(Z) \end{pmatrix}, \qquad
F'^{(n+1)}_{l,\mu+1/2,\nu}(Z) =
- \begin{pmatrix} \partial_{21} F'^{(n)}_{l,\mu,\nu}(Z) \\
\partial_{22} F^{(n)}_{l,\mu,\nu}(Z) \end{pmatrix},
$$
$$
G^{(n+1)}_{l,\mu,\nu}(Z) = \bigl( G^{(n)}_{l,\mu+1/2,\nu}(Z),
G^{(n)}_{l,\mu-1/2,\nu}(Z) \bigr),
$$
$$
G'^{(n+1)}_{l,\mu,\nu}(Z) = \bigl( G'^{(n)}_{l+1/2,\mu,\nu-1/2}(Z),
G'^{(n)}_{l+1/2,\mu,\nu+1/2}(Z) \bigr).
$$
\end{lem}

Next, we derive two expansions of the Cauchy-Fueter kernel for $n$-regular
functions (\ref{k_n-kernel}) in terms of these functions $F^{(n)}_{l,\mu,\nu}(Z)$,
$F'^{(n)}_{l,\mu,\nu}(Z)$, $G^{(n)}_{l,\mu,\nu}(Z)$ and $G'^{(n)}_{l,\mu,\nu}(Z)$.
This is an $n$-regular function analogue of Proposition 26 from \cite{FL1}
for the usual regular functions (see also Proposition 112 in \cite{FL4})
and Proposition 12 in \cite{FL4} for doubly regular functions.

\begin{prop}  \label{expansion-prop}
We have the following expansions
\begin{equation}  \label{k_n-expansion}
k_{n/2}(Z-W) = \sum_{l,\mu,\nu} F^{(n)}_{l,\mu,\nu}(W) \cdot G'^{(n)}_{l,\mu,\nu}(Z)
= \sum_{l,\mu,\nu} F'^{(n)}_{l.\mu,\nu}(Z) \cdot G^{(n)}_{l,\mu,\nu}(W),
\end{equation}
%\begin{equation}  \label{k_n-expansion-1}
%k_{n/2}(Z-W) = \sum_{l,\mu,\nu} F^{(n)}_{l,\mu,\nu}(W) \cdot G'^{(n)}_{l,\mu,\nu}(Z)
%\end{equation}
%and
%\begin{equation}  \label{k_n-expansion-2}
%k_{n/2}(Z-W) = \sum_{l,\mu,\nu} F'^{(n)}_{l.\mu,\nu}(Z) \cdot G^{(n)}_{l,\mu,\nu}(W),
%\end{equation}
which converge uniformly on compact subsets in the region
$\{ (Z,W) \in \HC^{\times} \times \HC; \: WZ^{-1} \in \BB D^+ \}$.
The sums are taken first over all $\mu=-l-n/2,-l,\dots,l+n/2$, 
$\nu=-l,-l+1,\dots,l$, then over $l=0,\frac12,1,\frac32,\dots$.
\end{prop}

\begin{proof}
We use induction on $n$. If $n=1$, the result reduces to
Proposition 26 in \cite{FL1} (see also Proposition 112 in \cite{FL4}).
To prove the inductive step, we use the recursive relations from
Lemma \ref{recursive-lem} to expand
$$
\sum_{l,\mu,\nu} F^{(n+1)}_{l,\mu,\nu}(W) \cdot G'^{(n+1)}_{l,\mu,\nu}(Z)
= \begin{pmatrix} A_{11} & A_{12} \\ A_{21} & A_{22} \end{pmatrix},
$$
where
\begin{align*}
A_{11} &= \sum_{l,\mu,\nu} (\partial_{11})_W F^{(n)}_{l+1/2,\mu,\nu-1/2}(W)
\cdot G'^{(n)}_{l+1/2,\mu,\nu-1/2}(Z)
= (\partial_{11})_W k_{n/2}(Z-W),  \\
A_{12} &= \sum_{l,\mu,\nu} (\partial_{21})_W F^{(n)}_{l+1/2,\mu,\nu+1/2}(W)
\cdot G'^{(n)}_{l+1/2,\mu,\nu+1/2}(Z)
= (\partial_{21})_W k_{n/2}(Z-W),  \\
A_{21} &= \sum_{l,\mu,\nu} (\partial_{12})_W F^{(n)}_{l+1/2,\mu,\nu-1/2}(W)
\cdot G'^{(n)}_{l+1/2,\mu,\nu-1/2}(Z)
= (\partial_{12})_W k_{n/2}(Z-W),  \\
A_{22} &= \sum_{l,\mu,\nu} (\partial_{22})_W F^{(n)}_{l+1/2,\mu,\nu+1/2}(W)
\cdot G'^{(n)}_{l+1/2,\mu,\nu+1/2}(Z)
= (\partial_{22})_W k_{n/2}(Z-W)
\end{align*}
by induction hypothesis. This proves
\begin{multline*}
\sum_{l,\mu,\nu} F^{(n+1)}_{l,\mu,\nu}(W) \cdot G'^{(n+1)}_{l,\mu,\nu}(Z)  \\
= \begin{pmatrix} (\partial_{11})_W k_{n/2}(Z-W) &
(\partial_{21})_W k_{n/2}(Z-W) \\ (\partial_{12})_W k_{n/2}(Z-W) &
(\partial_{22})_W k_{n/2}(Z-W) \end{pmatrix}
= k_{(n+1)/2}(Z-W).
\end{multline*}
The other expansion is proved similarly.
%$$
%\sum_{l,\mu,\nu} F'^{(n+1)}_{l,\mu,\nu}(Z) \cdot G^{(n+1)}_{l,\mu,\nu}(W)
%= \begin{pmatrix} B_{11} & B_{12} \\ B_{21} & B_{22} \end{pmatrix},
%$$
%where
%\begin{align*}
%B_{11} &= \sum_{l,\mu,\nu} -(\partial_{11})_Z F'^{(n)}_{l,\mu+1/2,\nu}(Z)
%\cdot G^{(n)}_{l,\mu+1/2,\nu}(W)
%= -(\partial_{11})_Z k_{n/2}(Z-W),  \\
%B_{12} &= \sum_{l,\mu,\nu} -(\partial_{21})_Z F'^{(n)}_{l,\mu-1/2,\nu}(Z)
%\cdot G^{(n)}_{l,\mu-1/2,\nu}(W)
%= -(\partial_{21})_Z k_{n/2}(Z-W),  \\
%B_{21} &= \sum_{l,\mu,\nu} -(\partial_{12})_Z F'^{(n)}_{l,\mu+1/2,\nu}(Z)
%\cdot G^{(n)}_{l,\mu+1/2,\nu}(W)
%= -(\partial_{12})_Z k_{n/2}(Z-W),  \\
%B_{22} &= \sum_{l,\mu,\nu} -(\partial_{22})_Z F'^{(n)}_{l,\mu-1/2,\nu}(Z)
%\cdot G^{(n)}_{l,\mu-1/2,\nu}(W)
%= -(\partial_{22})_Z k_{n/2}(Z-W).
%\end{align*}
\end{proof}

\section{$n$-Regular Functions on $\BB H^{\times}$}  \label{deg-inverse-section}

In this section we show that, if a (left or right) $n$-regular function is
defined on all of $\BB H^{\times}$, then the operators $(\deg+m)$, $m=2,\dots,n$,
can be inverted.
Thus the operator $\ddd = (\deg+n) \cdots (\deg+2)$ can be inverted as well.
This will be needed, for example, when we define the invariant bilinear pairing
for such functions.

We start with a left $n$-regular function
$f: \BB H^{\times} \to \BB S \odot \cdots \odot \BB S$
and derive some properties of such functions.
Of course, right $n$-regular functions
$g: \BB H^{\times} \to \BB S' \odot \cdots \odot \BB S'$ have similar properties.
Let $0<r<R$, then, by the Cauchy-Fueter formula for $n$-regular functions
(Theorem \ref{Fueter-n-reg}),
\begin{multline*}
\ddd f(W) =
\frac1{2\pi^2} \int_{S^3_R} k_{n/2}(Z-W) \cdot
(Dz \otimes Z \otimes \cdots \otimes Z) \cdot f(Z)\\
- \frac1{2\pi^2} \int_{S^3_r} k_{n/2}(Z-W) \cdot
(Dz \otimes Z \otimes \cdots \otimes Z) \cdot f(Z),
\end{multline*}
for all $W \in \BB H$ such that $r^2<N(W)<R^2$, where
$S^3_R \subset \BB H$ is the sphere of radius $R$ centered at the origin
$$
S^3_R = \{ X \in \BB H ;\: N(X)=R^2 \}.
$$
Define functions $\tilde f_+: \BB H \to \BB S \odot \cdots \odot \BB S$
and $\tilde f_-: \BB H^{\times} \to \BB S \odot \cdots \odot \BB S$ by
$$
\tilde f_+(W) =
\frac1{2\pi^2} \int_{S^3_R} k_{n/2}(Z-W) \cdot
(Dz \otimes Z \otimes \cdots \otimes Z) \cdot f(Z), \qquad R^2>N(W),
$$
$$
\tilde f_-(W) =
- \frac1{2\pi^2} \int_{S^3_r} k_{n/2}(Z-W) \cdot
(Dz \otimes Z \otimes \cdots \otimes Z) \cdot f(Z), \qquad r^2<N(W).
$$
Note that $\tilde f_+$ and $\tilde f_-$ are left $n$-regular on their
respective domains and that $\tilde f_-(W)$ decays at infinity at
a rate $\sim N(W)^{-1-n/2}$.

For a function $\phi$ defined on $\BB H$ or, slightly more generally,
on a star-shaped open subset of $\BB H$ centered at the origin, let
$$
\bigl( (\deg+m)^{-1} \phi \bigr)(Z) = \int_0^1 t^{m-1} \cdot \phi(tZ) \,dt,
\qquad m \ge 1,
$$
(cf. Subsection 2.4 of \cite{FL4}).
Similarly, for a function $\phi$ defined on $\BB H^{\times}$ and decaying
sufficiently fast at infinity, we can define $(\deg+m)^{-1} \phi$ as
$$
\bigl( (\deg+m)^{-1} \phi \bigr)(Z)
= - \int_1^{\infty} t^{m-1} \cdot \phi(tZ) \,dt.
$$
Then
$$
(\deg+m) \bigl( (\deg+m)^{-1} \phi \bigr)
= (\deg+m)^{-1} \bigl( (\deg+m) \phi \bigr) = \phi
$$
for functions $\phi$ that are either defined on star-shaped open subsets of
$\BB H$ centered at the origin or on $\BB H^{\times}$ and decaying sufficiently
fast at infinity. (In the same fashion one can also define $(\deg+m)^{-1}\phi$
for functions defined on star-shaped open subsets of $\HC$ centered at
the origin or on $\HC^{\times}$ and decaying sufficiently fast at infinity.)

We introduce functions
$$
f_+ = \ddd^{-1} \tilde f_+ \qquad \text{and} \qquad
f_- = \ddd^{-1} \tilde f_-.
$$

\begin{prop}  \label{f+f-prop}
Let $f: \BB H^{\times} \to \BB S \odot \cdots \odot \BB S$ be a left $n$-regular
function. Then $f(X)=f_+(X)+f_-(X)$, for all $X \in \BB H^{\times}$.
\end{prop}

\begin{proof}
The proof is the same as that of Proposition 13 in \cite{FL4}.
Let $f_* = f-f_+-f_-$, we want to show that $f_* \equiv 0$.
Note that $f_*: \BB H^{\times} \to \BB S \odot \cdots \odot \BB S$ is a left
$n$-regular function such that $\ddd f_* \equiv 0$, hence $f_*$ is a sum
of homogeneous functions of degrees $-2,-3,\dots,-n$:
$$
f_*=f_{-2}+f_{-3}+\cdots+f_{-n}, \qquad (\deg+j)f_{-j}=0.
$$
Since the operators $\underset{k}{\nabla^+}$, $k=1,\dots,n$ lower the degree
of homogeneity by one, each
$f_{-j}: \BB H^{\times} \to \BB S \odot \cdots \odot \BB S$ is left $n$-regular.
We will show that each $f_{-j}$, $j=2,3,\dots,n$, is identically zero.

Let
$$
f_{-n-1} = \frac{\partial}{\partial x^0} f_{-n}, \qquad
f_{-n-1}: \BB H^{\times} \to \BB S \odot \cdots \odot \BB S,
$$
then $f_{-n-1}$ is a left $n$-regular function that is homogeneous
of degree $-(n+1)$.

By the Cauchy-Fueter formulas for $n$-regular functions
(Theorem \ref{Fueter-n-reg}),
\begin{multline*}
(-1)^{n-1} (n-1)! f_{-n-1}(X) =
\frac1{2\pi^2} \int_{S^3_R} k_{n/2}(Z-X) \cdot
(Dz \otimes Z \otimes \cdots \otimes Z) \cdot f_{-n-1}(Z)  \\
- \frac1{2\pi^2} \int_{S^3_r} k_{n/2}(Z-X) \cdot
(Dz \otimes Z \otimes \cdots \otimes Z) \cdot f_{-n-1}(Z),
\end{multline*}
where $R,r>0$ are such that $r^2<N(X)<R^2$.
By Liouville's theorem (Corollary \ref{Liouville}), the first integral defines
a left $n$-regular function on $\BB H$ that is either constant or unbounded.
On the other hand, the second integral defines a left $n$-regular function
on $\BB H^{\times}$ that decays at infinity at a rate $\sim N(W)^{-1-n/2}$.
We conclude that $f_{-n-1} \equiv 0$, hence $f_{-n} \equiv 0$ as well.

A similar argument combined with induction shows that
$f_{-n+1}, \dots, f_{-2} \equiv 0$. Hence $f_* \equiv 0$.
\end{proof}

\begin{df}  \label{deg-inv-def}
Let $f: \BB H^{\times} \to \BB S \odot \cdots \odot \BB S$ be a left $n$-regular
function. We define
$$
(\deg+m)^{-1}f = (\deg+m)^{-1}f_+ + (\deg+m)^{-1}f_-,
\qquad m=1,2,\dots,n,
$$
then
$$
\ddd^{-1} f = \ddd^{-1} f_+ + \ddd^{-1} f_-.
$$
Similarly, we can define $(\deg+m)^{-1}g$ and $\ddd^{-1}g$ for right
$n$-regular functions $g: \BB H^{\times} \to \BB S' \odot \cdots \odot \BB S'$.
\end{df}

From the previous discussion we immediately obtain:

\begin{prop}  \label{ddd-inverse-prop}
Let $f: \BB H^{\times} \to \BB S \odot \cdots \odot \BB S$ be a left $n$-regular
function and $g: \BB H^{\times} \to \BB S' \odot \cdots \odot \BB S'$ a right
$n$-regular function. Then, for $m=1,2,\dots,n$,
\begin{align*}
(\deg+m) \bigl( (\deg+m)^{-1} f \bigr) = (\deg+m)^{-1} \bigl((\deg+m) f \bigr)
&= f, \\
(\deg+m) \bigl( (\deg+m)^{-1} g \bigr) = (\deg+m)^{-1} \bigl( (\deg+m) g \bigr)
&= g, \\
\ddd \circ \ddd^{-1} f = \ddd^{-1} \circ \ddd f &= f, \\
\ddd \circ \ddd^{-1} g = \ddd^{-1} \circ \ddd g &= g.
\end{align*}
\end{prop}

From the expansions of the Cauchy-Fueter kernel (\ref{k_n-expansion})
we immediately obtain an analogue of Laurent series
expansion for $n$-regular functions.

\begin{cor}  \label{Laurent-expansion}
Let $f: \BB H^{\times} \to \BB S \odot \cdots \odot \BB S$ be a left $n$-regular
function, write $f=f_++f_-$ as in Proposition \ref{f+f-prop}.
Then the functions $f_+$ and $f_-$ can be expanded as series
$$
f_+(X) = \sum_l \Bigl( \sum_{\mu,\nu} a_{l,\mu,\nu} F^{(n)}_{l,\mu,\nu}(X) \Bigr),
\qquad
f_-(X) = \sum_l \Bigl( \sum_{\mu,\nu} b_{l,\mu,\nu} F'^{(n)}_{l,\mu,\nu}(X) \Bigr).
$$

If $g: \BB H^{\times} \to \BB S' \odot \cdots \odot \BB S'$ is a right $n$-regular
function, then it can be expressed as $g=g_+ + g_-$ in a similar way,
and the functions $g_+$ and $g_-$ can be expanded as series
$$
g_+(X) = \sum_l \Bigl( \sum_{\mu,\nu} c_{l,\mu,\nu} G^{(n)}_{l,\mu,\nu}(X) \Bigr),
\qquad
g_-(X) = \sum_l \Bigl( \sum_{\mu,\nu} d_{l,\mu,\nu} G'^{(n)}_{l,\mu,\nu}(X) \Bigr).
$$
\end{cor}

Formulas expressing the coefficients $a_{l,\mu,\nu}$, $b_{l,\mu,\nu}$,
$c_{l,\mu,\nu}$ and $d_{l,\mu,\nu}$ will be given in Corollary \ref{Laurent-coeff}.

\section{Invariant Bilinear Pairing for $n$-Regular Functions}  \label{bilinear-pairing-section}

We define a pairing between left and right $n$-regular functions as follows.
If $f(Z)$ and $g(Z)$ are left and right $n$-regular functions on
$\BB H^{\times}$ respectively, then, by the results of the previous section,
$\ddd^{-1}f$ and $\ddd^{-1}g$ are well defined, and we set
\begin{equation}   \label{pairing1}
\langle f, g \rangle_{{\cal R}_n} = \frac1{2\pi^2}
\int_{Z \in S^3_R} g(Z) \cdot (Dz \otimes Z \otimes \cdots \otimes Z) \cdot
\bigl( \ddd^{-1} f \bigr)(Z),
\end{equation}
where $S^3_R \subset \BB H$ is the sphere of radius $R$
centered at the origin
$$
S^3_R = \{ X \in \BB H ;\: N(X)=R^2 \}.
$$
Note that we use the subscript ${\cal R}_n$ in $\langle f, g \rangle_{{\cal R}_n}$
to indicate that the pairing is between left and right $n$-regular functions;
however, this is not a pairing between the spaces ${\cal R}_n$ and
${\cal R}'_n$, since the functions $f$ and $g$ are not allowed to have
singularities away from the origin.
Recall that by Lemma 6 in \cite{FL1} the $3$-form $Dz$ restricted to $S^3_R$
becomes $Z\,dS/R$, where $dS$ is the usual Euclidean volume element on $S^3_R$.
Thus we can rewrite (\ref{pairing1}) as
\begin{align*}
\langle f, g \rangle_{{\cal R}_n} &= \frac1{2\pi^2} \int_{Z \in S^3_R} g(Z) \cdot
(Z \otimes \cdots \otimes Z) \cdot \bigl( \ddd^{-1} f \bigr)(Z) \,\frac{dS}R\\
&= \frac1{2\pi^2} \int_{Z \in S^3_R} g(Z) \cdot (Z \otimes \cdots
\underset{\text{$k$-th place}}{\otimes Dz \otimes} \cdots \otimes Z)
\cdot \bigl( \ddd^{-1} f \bigr)(Z), \qquad k=1,\dots,n.
\end{align*}
Since $\underset{k}{\nabla^+} \ddd^{-1} f =0$ and, by Lemma \ref{zf-regular},
$$
\Bigl[ g(Z) (Z \otimes \cdots \underset{\text{$k$-th place}}{\otimes 1 \otimes}
\cdots \otimes Z) \Bigr] \overleftarrow{\underset{k}{\nabla^+}} =0,
$$
the integrand of (\ref{pairing1}) is a closed $3$-form. Thus,
by Lemma \ref{Cauchy-thm}, the contour of integration can be continuously
deformed. In particular, this pairing does not depend on the choice of $R>0$.

\begin{prop}
If $f(Z)$ and $g(Z)$ are left and right $n$-regular functions respectively
on $\BB H^{\times}$, then
\begin{multline}  \label{deg-switch}
\langle f, g \rangle_{{\cal R}_n} = \frac1{2\pi^2} \int_{Z \in S^3_R} g(Z) \cdot
(Dz \otimes Z \otimes \cdots \otimes Z) \cdot \bigl( \ddd^{-1} f \bigr)(Z) \\
= \frac{(-1)^{n-1}}{2\pi^2} \int_{Z \in S^3_R} \bigl( \ddd^{-1} g \bigr)(Z)
\cdot (Dz \otimes Z \otimes \cdots \otimes Z) \cdot f(Z).
\end{multline}
\end{prop}

\begin{proof}
Since the expression
$$
\int_{Z \in S^3_R} g(Z) \cdot (Dz \otimes Z \otimes \cdots \otimes Z) \cdot f(Z)
$$
is independent of the choice of $R>0$, we have:
\begin{multline*}
0 = \frac{d}{dt} \biggr|_{t=1} \biggl( \int_{Z \in  S^3_{tR}} g(Z)
\cdot (Dz \otimes Z \otimes \cdots \otimes Z) \cdot f(Z) \biggr)  \\
= \int_{Z \in  S^3_R} \Bigl( \bigl( (\deg+m)g \bigr)(Z) \cdot
(Dz \otimes Z \otimes \cdots \otimes Z) \cdot f(Z)  \\
+ g(Z) \cdot (Dz \otimes Z \otimes \cdots \otimes Z) \cdot
\bigl( (\deg+2+n-m)f \bigr)(Z) \Bigr),
\end{multline*}
for all $m \in \BB Z$. From this (\ref{deg-switch}) follows.
\end{proof}

\begin{cor}
If $f(Z)$ and $g(Z)$ are left and right $n$-regular functions on $\HC$
respectively and $W \in \BB D_R^+$ (open domains $\BB D_R^{\pm}$ were defined by
equation (22) in \cite{FL3}), the Cauchy-Fueter formulas for $n$-regular
functions (Theorem \ref{Fueter-n-reg}) can be rewritten as
$$
f(W) = \bigl\langle k_{n/2}(Z-W), f(Z) \bigr\rangle_{{\cal R}_n}
\qquad \text{and} \qquad
g(W) = (-1)^{n-1} \bigl\langle g(Z), k_{n/2}(Z-W) \bigr\rangle_{{\cal R}_n}.
$$
\end{cor}

We can rewrite the bilinear pairing (\ref{pairing1})
in a more symmetrical way.
Let $0<r<R$ and $0<r_1<R<r_2$.  Using the Cauchy-Fueter formulas for
$n$-regular functions (Theorem \ref{Fueter-n-reg}), substituting
\begin{multline*}
g(Z) = \frac1{2\pi^2} \int_{W \in S^3_{r_2}} \bigl( \ddd^{-1} g \bigr)(W) \cdot
(Dw \otimes W \otimes \cdots \otimes W) \cdot k_{n/2}(W-Z) \\
- \frac1{2\pi^2} \int_{W \in S^3_{r_1}} \bigl( \ddd^{-1} g \bigr)(W) \cdot
(Dw \otimes W \otimes \cdots \otimes W) \cdot k_{n/2}(W-Z),
\qquad Z \in \BB D^+_{r_2} \cap \BB D^-_{r_1},
\end{multline*}
into (\ref{pairing1}) and shifting contours of integration, we obtain:
\begin{multline}  \label{pairing1-symm}
4\pi^4 \cdot \langle f, g \rangle_{{\cal R}_n} =  \\
\iint_{\genfrac{}{}{0pt}{}{Z \in S^3_r}{W \in S^3_R}} \bigl( \ddd^{-1} g \bigr)(W)
\cdot (Dw \otimes W \otimes \cdots \otimes W) \cdot k_{n/2}(W-Z) \cdot
(Dz \otimes Z \otimes \cdots \otimes Z) \cdot \bigl( \ddd^{-1} f \bigr)(Z)  \\
- \iint_{\genfrac{}{}{0pt}{}{Z \in S^3_R}{W \in S^3_r}}
\bigl( \ddd^{-1} g \bigr)(W) \cdot (Dw \otimes W \otimes \cdots \otimes W)
\cdot k_{n/2}(W-Z) \cdot (Dz \otimes Z \otimes \cdots \otimes Z) \cdot
\bigl( \ddd^{-1} f \bigr)(Z).
\end{multline}

As usual, we realize $U(2) \times U(2)$ as a diagonal subgroup of $GL(2,\HC)$:
\begin{equation}  \label{U(2)xU(2)}
U(2) \times U(2) = \bigl\{
\bigl(\begin{smallmatrix} a & 0 \\ 0 & d \end{smallmatrix}\bigr)
\in GL(2,\HC); \: a,d \in \HC,\: a^*a=1,\: d^*d=1 \bigr\},
\end{equation}
where $a^*$ and $d^*$ denote the matrix adjoints of $a$ and $d$ in the
standard realization of $\HC$ as $2 \times 2$ complex matrices.
Then $SU(2) \times SU(2)$ is realized as
\begin{equation}  \label{SU(2)xSU(2)}
SU(2) \times SU(2) = \bigl\{
\bigl(\begin{smallmatrix} a & 0 \\ 0 & d \end{smallmatrix}\bigr)
\in GL(2,\HC); \: a,d \in \BB H,\: N(a)=N(d)=1 \bigr\}.
\end{equation}

\begin{prop}  \label{invariant}
The bilinear pairing (\ref{pairing1}) is $SU(2) \times SU(2)$ and
$\mathfrak{gl}(2,\HC)$-invariant.
\end{prop}

\begin{proof}
It is sufficient to show that the pairing is invariant under
$$
\BB H^{\times} \times \BB H^{\times} = \bigl\{
\bigl(\begin{smallmatrix} a & 0 \\ 0 & d \end{smallmatrix}\bigr) ;\:
a, d \in \BB H^{\times} \bigr\} \subset GL(2,\HC),
$$
$\bigl(\begin{smallmatrix} 0 & B \\ 0 & 0 \end{smallmatrix}\bigr)
\in \mathfrak{gl}(2,\HC)$, $B \in \HC$,
and inversion
$\bigl(\begin{smallmatrix} 0 & 1 \\ 1 & 0 \end{smallmatrix}\bigr)
\in GL(2,\HC)$.

First, let
$h= \bigl(\begin{smallmatrix} a & 0 \\ 0 & d \end{smallmatrix}\bigr)
\in GL(2,\HC)$, $a,d \in \BB H$, $\tilde Z = a^{-1}Zd$.
Recall that actions $\pi_{nl}$ and $\pi_{nr}$ are described by equations
(\ref{pi_nl})-(\ref{pi_nr}).
Using Proposition 11 from \cite{FL1} we obtain:
\begin{multline*}
2\pi^2 \cdot \langle \pi_{nl}(h)f, \pi_{nr}(h)g \rangle_{{\cal R}_n}
= \int_{Z \in S^3_R}
(\pi_{nr}(h)g)(Z) \cdot (Dz \otimes Z \otimes \cdots \otimes Z) \cdot
\bigl( \ddd^{-1} (\pi_{nl}(h)f) \bigr)(Z)  \\
= \int_{Z \in S^3_R} g(a^{-1}Zd) \cdot
\frac{a^{-1} \otimes \cdots \otimes a^{-1}}{N(a)}
\cdot (Dz \otimes Z \otimes \cdots \otimes Z) \cdot \ddd^{-1} \biggl(
\frac{d \otimes \cdots \otimes d}{N(d)^{-1}} \cdot f(a^{-1}Zd) \biggr)  \\
= \int_{\tilde Z \in S^3_{R'}} g(\tilde Z) \cdot
(D\tilde z \otimes \tilde Z \otimes \cdots \otimes \tilde Z)
\cdot \bigl( \ddd^{-1} f \bigr)(\tilde Z)
= 2\pi^2 \cdot \langle f,g \rangle_{{\cal R}_n},
\end{multline*}
where $R' = \sqrt{N(a)^{-1} \cdot N(d)} \cdot R$.

Next, we recall that matrices
$\bigl(\begin{smallmatrix} 0 & B \\ 0 & 0 \end{smallmatrix}\bigr)
\in \mathfrak{gl}(2,\HC)$, $B \in \HC$,
act by differentiation (Lemma \ref{Lie-alg-action}).
For example, if
$B= \bigl(\begin{smallmatrix} 1 & 0 \\ 0 & 0 \end{smallmatrix}\bigr) \in \HC$,
using expressions (\ref{deg-switch})-(\ref{pairing1-symm})
for the bilinear pairing and the symmetry relation (\ref {k_n-symmetry}),
we obtain:
%using the symmetric expression (\ref{pairing1-symm}) we obtain:
\begin{multline*}
4\pi^4 \cdot \Bigl\langle \pi_{nl}\bigl(\begin{smallmatrix} 0 & B \\
0 & 0 \end{smallmatrix}\bigr)f, g \Bigr\rangle_{{\cal R}_n}
= (-1)^{n-1} 2\pi^2 \int_{W \in S^3_R} \bigl( \ddd^{-1} g \bigr)(W)
\cdot (Dw \otimes W \otimes \cdots \otimes W) \cdot
\frac{\partial f}{\partial w_{11}} (W)  \\
%= \iint_{\genfrac{}{}{0pt}{}{Z \in S^3_r}{W \in S^3_R}} \bigl( \ddd^{-1} g \bigr)(W)
%\cdot (Dw \otimes W \otimes \cdots \otimes W) \cdot k_{n/2}(W-Z) \cdot
%(Dz \otimes Z \otimes \cdots \otimes Z) \cdot
%\biggl(\ddd^{-1} \frac{\partial f}{\partial z_{11}} \biggr)(Z)  \\
%- \iint_{\genfrac{}{}{0pt}{}{Z \in S^3_R}{W \in S^3_r}} \bigl( \ddd^{-1} g \bigr)(W)
%\cdot (Dw \otimes W \otimes \cdots \otimes W) \cdot k_{n/2}(W-Z) \cdot
%(Dz \otimes Z \otimes \cdots \otimes Z) \cdot
%\biggl(\ddd^{-1} \frac{\partial f}{\partial z_{11}} \biggr)(Z)  \\
= \iint_{\genfrac{}{}{0pt}{}{Z \in S^3_r}{W \in S^3_R}} \bigl( \ddd^{-1} g \bigr)(W)
\cdot (Dw \otimes W \otimes \cdots \otimes W) \cdot
\biggl( \frac{\partial}{\partial w_{11}} k_{n/2}(W-Z) \biggr) \cdot
(Dz \otimes Z \otimes \cdots \otimes Z) \cdot \bigl(\ddd^{-1} f \bigr)(Z)  \\
- \iint_{\genfrac{}{}{0pt}{}{Z \in S^3_R}{W \in S^3_r}} \bigl( \ddd^{-1} g \bigr)(W)
\cdot (Dw \otimes W \otimes \cdots \otimes W) \cdot
\biggl( \frac{\partial}{\partial w_{11}} k_{n/2}(W-Z) \biggr) \cdot
(Dz \otimes Z \otimes \cdots \otimes Z) \cdot \bigl(\ddd^{-1} f \bigr)(Z)  \\
= - \iint_{\genfrac{}{}{0pt}{}{Z \in S^3_r}{W \in S^3_R}} \bigl( \ddd^{-1} g \bigr)(W)
\cdot (Dw \otimes W \otimes \cdots \otimes W) \cdot
\biggl( \frac{\partial}{\partial z_{11}} k_{n/2}(W-Z) \biggr) \cdot
(Dz \otimes Z \otimes \cdots \otimes Z) \cdot \bigl(\ddd^{-1} f \bigr)(Z)  \\
+ \iint_{\genfrac{}{}{0pt}{}{Z \in S^3_R}{W \in S^3_r}} \bigl( \ddd^{-1} g \bigr)(W)
\cdot (Dw \otimes W \otimes \cdots \otimes W) \cdot
\biggl( \frac{\partial}{\partial z_{11}} k_{n/2}(W-Z) \biggr) \cdot
(Dz \otimes Z \otimes \cdots \otimes Z) \cdot \bigl(\ddd^{-1} f \bigr)(Z)  \\
%= - \iint_{\genfrac{}{}{0pt}{}{Z \in S^3_r}{W \in S^3_R}}
%\biggl( \ddd^{-1} \frac{\partial g}{\partial w_{11}} \biggr)(W)
%\cdot (Dw \otimes W \otimes \cdots \otimes W) \cdot k_{n/2}(W-Z) \cdot
%(Dz \otimes Z \otimes \cdots \otimes Z) \cdot \bigl(\ddd^{-1} f \bigr)(Z)  \\
%+ \iint_{\genfrac{}{}{0pt}{}{Z \in S^3_R}{W \in S^3_r}}
%\biggl( \ddd^{-1} \frac{\partial g}{\partial w_{11}} \biggr)(W)
%\cdot (Dw \otimes W \otimes \cdots \otimes W) \cdot k_{n/2}(W-Z) \cdot
%(Dz \otimes Z \otimes \cdots \otimes Z) \cdot \bigl(\ddd^{-1} f \bigr)(Z)  \\
= -2\pi^2 \int_{Z \in S^3_R} \frac{\partial g}{\partial z_{11}}(Z) \cdot
(Dz \otimes Z \otimes \cdots \otimes Z) \cdot \bigl(\ddd^{-1} f \bigr)(Z)
= - 4\pi^4 \cdot \Bigl\langle f, \pi_{nr}\bigl(\begin{smallmatrix} 0 & B \\
0 & 0 \end{smallmatrix}\bigr)g \Bigr\rangle_{{\cal R}_n}.
\end{multline*}

Finally, if
$h = \bigl(\begin{smallmatrix} 0 & 1 \\ 1 & 0 \end{smallmatrix}\bigr)
\in GL(2,\HC)$, changing the variable to $\tilde Z = Z^{-1}$ -- which is
an orientation reversing map $S^3_R \to S^3_{1/R}$ -- and using
Proposition 11 from \cite{FL1}, we have:
\begin{multline*}
2\pi^2 \cdot \langle \pi_{nl}(h)f, \pi_{nr}(h)g \rangle_{{\cal R}_n} \\
= (-1)^n \int_{Z \in S^3_R}
g(Z^{-1}) \cdot \frac{Z^{-1} \otimes \cdots \otimes Z^{-1}}{N(Z)} \cdot
(Dz \otimes Z \otimes \cdots \otimes Z) \cdot \ddd^{-1}
\biggl( \frac{Z^{-1} \otimes \cdots \otimes Z^{-1}}{N(Z)} \cdot f(Z^{-1}) \biggr)\\
= - \int_{Z \in S^3_R} g(Z^{-1}) \cdot
\frac{(Z^{-1} \cdot Dz \cdot Z^{-1}) \otimes Z^{-1} \otimes \cdots \otimes Z^{-1}}
{N(Z)^2} \cdot \bigl( \ddd^{-1} f \bigr)(Z^{-1}) \\
= - \int_{\tilde Z \in S^3_{1/R}} g(\tilde Z) \cdot
(D\tilde z \otimes \tilde Z \otimes \cdots \otimes \tilde Z) \cdot
\bigl( \ddd^{-1} f \bigr)(\tilde Z)
= - 2\pi^2 \cdot \langle f, g \rangle_{{\cal R}_n},
\end{multline*}
where for the second equality we used the property
\begin{multline*}
(\deg+m)^{-1}
\biggl( \frac{Z^{-1} \otimes \cdots \otimes Z^{-1}}{N(Z)} \cdot f(Z^{-1}) \biggr)\\
= - \frac{Z^{-1} \otimes \cdots \otimes Z^{-1}}{N(Z)}
\cdot \bigl( (\deg+n+2-m)^{-1} f \bigr)(Z^{-1}).
\end{multline*}
(Note that the negative sign in
$\langle \pi_{nl}(h)f, \pi_{nr}(h)g \rangle_{{\cal R}_n}
= - \langle f, g \rangle_{{\cal R}_n}$
does not affect the invariance of the bilinear pairing under the Lie algebra
$\mathfrak{gl}(2,\HC)$.)
%(Note that
%$$
%\ddd^{-1} \bigl( N(Z)^{-1} \cdot (Z^{-1} \otimes \cdots \otimes Z^{-1}) \cdot
%f(Z^{-1}) \bigr) 
%= (-1)^{n-1} N(Z)^{-1} \cdot (Z^{-1} \otimes \cdots \otimes Z^{-1}) \cdot
%(\ddd^{-1} f)(Z^{-1}),
%$$
%another negative sign appears when passing from $Dz$ to $D\tilde z$
%and a third negative sign appears because the map $Z \mapsto Z^{-1}$
%reverses the orientation of $3$-spheres.)
\end{proof}

Next we describe orthogonality relations for $n$-regular functions.
Recall functions $F^{(n)}_{l,\mu,\nu}$, $F'^{(n)}_{l,\mu,\nu}$,
$G^{(n)}_{l,\mu,\nu}$ and $G'^{(n)}_{l,\mu,\nu}$ introduced
in Section \ref{expansion-section}, these are the functions that appear in
matrix coefficient expansions of the Cauchy-Fueter kernel (\ref{k_n-expansion}).

\begin{prop}  \label{orthogonality-nreg-prop}
We have the following orthogonality relations:
\begin{equation}  \label{n-orthogonality1}
\langle F^{(n)}_{l,\mu,\nu}, G'^{(n)}_{l',\mu',\nu'} \rangle_{{\cal R}_n} =
(-1)^{n-1} \langle F'^{(n)}_{l,\mu,\nu}, G^{(n)}_{l',\mu',\nu'} \rangle_{{\cal R}_n} =
\delta_{ll'} \cdot \delta_{\mu\mu'} \cdot \delta_{\nu\nu'},
\end{equation}
\begin{equation}  \label{n-orthogonality2}
\langle F^{(n)}_{l,\mu,\nu}, G^{(n)}_{l',\mu',\nu'} \rangle_{{\cal R}_n}
= \langle F'^{(n)}_{l,\mu,\nu}, G'^{(n)}_{l',\mu',\nu'} \rangle_{{\cal R}_n} = 0.
\end{equation}
In particular, (\ref{pairing1}) is a non-degenerate bilinear pairing between
$$
\BB C\text{-span of } \bigl\{ F^{(n)}_{l,\mu,\nu}(Z) \bigr\}
\quad \text{and} \quad
\BB C\text{-span of } \bigl\{ G'^{(n)}_{l,\mu,\nu}(Z) \bigr\}
$$
and between
$$
\BB C\text{-span of } \bigl\{ F'^{(n)}_{l,\mu,\nu}(Z) \bigr\}
\quad \text{and} \quad
\BB C\text{-span of } \bigl\{ G^{(n)}_{l,\mu,\nu}(Z) \bigr\},
$$
%$$
%\BB C\text{-span of } \bigl\{ F^{(n)}_{l,\mu,\nu}(Z) \bigr\} \oplus
%\BB C\text{-span of } \bigl\{ F'^{(n)}_{l,\mu,\nu}(Z) \bigr\}
%$$
%and
%$$
%\BB C\text{-span of } \bigl\{ G^{(n)}_{l,\mu,\nu}(Z) \bigr\} \oplus
%\BB C\text{-span of } \bigl\{ G'^{(n)}_{l,\mu,\nu}(Z) \bigr\},
%$$
where $l = 0, \frac12, 1, \frac32, \dots$,
$\mu \in \BB Z +l+n/2$,
$\nu \in \BB Z +l$,
$-l-n/2 \le \mu \le l+n/2$, $-l \le \nu \le l$.
\end{prop}

\begin{proof}
Observe that the functions $F^{(n)}_{l,\mu,\nu}$ and $G^{(n)}_{l,\mu,\nu}$
are homogeneous of degree $2l$, while the functions $F'^{(n)}_{l,\mu,\nu}$
and $G'^{(n)}_{l,\mu,\nu}$ are homogeneous of degree $-(2l+n+2)$.
On the one hand, the bilinear pairing (\ref{pairing1}) is independent
of the choice of radius $R>0$.
Recall that, by Lemma 6 in \cite{FL1}, the $3$-form $Dz$ restricted to $S^3_R$
equals $Z\,dS/R$. It follows that if $f$ is homogeneous of degree $d_f$
and $g$ is homogeneous of degree $d_g$, then either
$\langle f,g \rangle_{{\cal R}_n} =0$ or $d_f+d_g=-(n+2)$.
This proves (\ref{n-orthogonality2}) and
$$
\langle F^{(n)}_{l,\mu,\nu}, G'^{(n)}_{l',\mu',\nu'} \rangle_{{\cal R}_n} =
\langle F'^{(n)}_{l,\mu,\nu}, G^{(n)}_{l',\mu',\nu'} \rangle_{{\cal R}_n} = 0
\qquad \text{if $l \ne l'$}.
$$
Then (\ref{n-orthogonality1}) follows from the expansions (\ref{k_n-expansion}),
the Cauchy-Fueter formulas (Theorem \ref{Fueter-n-reg})
and equation (\ref{deg-switch}).
\end{proof}

\begin{cor}  \label{Laurent-coeff}
The coefficients $a_{l,\mu,\nu}$, $b_{l,\mu,\nu}$, $c_{l,\mu,\nu}$ and $d_{l,\mu,\nu}$
of Laurent expansions of $n$-regular functions given in
Corollary \ref{Laurent-expansion} are given by the following expressions:
\begin{center}
\begin{tabular}{lcl}
$a_{l.\mu,\nu} = \langle f, G'^{(n)}_{l,\mu,\nu} \rangle_{{\cal R}_n}$, & \qquad &
$b_{l.\mu,\nu} = (-1)^{n-1} \langle f, G^{(n)}_{l,\mu,\nu} \rangle_{{\cal R}_n}$, \\
$c_{l.\mu,\nu} = (-1)^{n-1} \langle F'^{(n)}_{l,\mu,\nu}, g \rangle_{{\cal R}_n}$,
& \qquad & $d_{l.\mu,\nu} = \langle F^{(n)}_{l,\mu,\nu}, g \rangle_{{\cal R}_n}$.
\end{tabular}
\end{center}
\end{cor}

\section{Spaces of $n$-Regular Functions as Representations of
  the Conformal Lie Algebra}  \label{Harish-Chandra-section}

In this section we identify the irreducible components and the $K$-types
of the spaces of left and right $n$-regular functions on $\BB H^{\times}$
regarded as representations of the conformal Lie algebra $\mathfrak{gl}(2,\HC)$.
Let
\begin{equation}
{\cal F}_n^+ = \BB C\text{-span of } \bigl\{ F^{(n)}_{l,\mu,\nu}(Z) \bigr\},
\qquad
{\cal F}_n^- = \BB C\text{-span of } \bigl\{ F'^{(n)}_{l,\mu,\nu}(Z) \bigr\},
\end{equation}
\begin{equation}
{\cal G}_n^+ = \BB C\text{-span of } \bigl\{ G^{(n)}_{l,\mu,\nu}(Z) \bigr\},
\qquad
{\cal G}_n^- = \BB C\text{-span of } \bigl\{ G'^{(n)}_{l,\mu,\nu}(Z) \bigr\},
\end{equation}
where
$$
l = 0, \frac12, 1, \frac32, \dots, \quad
\mu \in \BB Z +l+n/2, \quad \nu \in \BB Z +l, \quad
-l-n/2 \le \mu \le l+n/2, \quad -l \le \nu \le l.
$$

We can think of
$\BB C[z^0,z^1,z^2,z^3] \otimes (\BB S \odot \cdots \odot \BB S)$
as the space of $(\BB S \odot \cdots \odot \BB S)$-valued polynomials on
$\BB H$ or $\HC$.
Similarly,
$\BB C[z^0,z^1,z^2,z^3,N(Z)^{-1}] \otimes (\BB S' \odot \cdots \odot \BB S')$
can be thought as $(\BB S' \odot \cdots \odot \BB S')$-valued Laurent
polynomials on $\HC^{\times}$.
We can give basis-free descriptions of the spaces
${\cal F}_n^{\pm}$ and ${\cal G}_n^{\pm}$:

\begin{prop}  \label{basis-free-description}
We have:
\begin{align*}
{\cal F}_n^+ &= \{ \text{left $n$-regular polynomials }
f \in \BB C[z^0,z^1,z^2,z^3] \otimes (\BB S \odot \cdots \odot \BB S) \},  \\
{\cal G}_n^+ &= \{ \text{right $n$-regular polynomials }
g \in \BB C[z^0,z^1,z^2,z^3] \otimes (\BB S' \odot \cdots \odot \BB S') \},  \\
{\cal F}_n^+ \oplus {\cal F}_n^- &= \{ \text{left $n$-regular functions }
f \in \BB C[z^0,z^1,z^2,z^3,N(Z)^{-1}] \otimes (\BB S \odot \cdots \odot \BB S)
\},  \\
{\cal G}_n^+ \oplus {\cal G}_n^- &= \{ \text{right $n$-regular functions }
g \in \BB C[z^0,z^1,z^2,z^3,N(Z)^{-1}] \otimes (\BB S' \odot \cdots \odot \BB S')
\},  \\
{\cal F}_n^- &= \{ \text{left $n$-regular functions }
f \in \BB C[z^0,z^1,z^2,z^3,N(Z)^{-1}] \otimes (\BB S \odot \cdots \odot \BB S);\\
& \hskip1in
\pi_{nl} \bigl(\begin{smallmatrix} 0 & 1 \\ 1 & 0 \end{smallmatrix}\bigr) 
f \in \BB C[z^0,z^1,z^2,z^3] \otimes (\BB S \odot \cdots \odot \BB S) \},  \\
{\cal G}_n^- &= \{ \text{right $n$-regular functions }
g \in \BB C[z^0,z^1,z^2,z^3,N(Z)^{-1}] \otimes (\BB S' \odot \cdots \odot \BB S');\\
& \hskip1in
\pi_{nr} \bigl(\begin{smallmatrix} 0 & 1 \\ 1 & 0 \end{smallmatrix}\bigr) 
g \in \BB C[z^0,z^1,z^2,z^3] \otimes (\BB S' \odot \cdots \odot \BB S') \}.
\end{align*}
In particular, we have complex vector space isomorphisms
$$
\pi_{nl}\bigl(\begin{smallmatrix} 0 & 1 \\ 1 & 0 \end{smallmatrix}\bigr):
{\cal F}_n^+ \longleftrightarrow {\cal F}_n^-
\qquad \text{and} \qquad
\pi_{nr}\bigl(\begin{smallmatrix} 0 & 1 \\ 1 & 0 \end{smallmatrix}\bigr):
{\cal G}_n^+ \longleftrightarrow {\cal G}_n^-.
$$
\end{prop}

\begin{proof}
The descriptions of ${\cal F}_n^+$, ${\cal G}_n^+$,
${\cal F}_n^+ \oplus {\cal F}_n^-$ and ${\cal G}_n^+ \oplus {\cal G}_n^-$
follow from Corollaries \ref{Laurent-expansion} and \ref{Laurent-coeff}.
Note that the inversion
$\bigl(\begin{smallmatrix} 0 & 1 \\ 1 & 0 \end{smallmatrix}\bigr) \in GL(2,\HC)$
preserves
$\BB C[z^0,z^1,z^2,z^3,N(Z)^{-1}] \otimes (\BB S \odot \cdots \odot \BB S)$
as well as ${\cal F}_n^+ \oplus {\cal F}_n^-$.
Comparing the degrees of homogeneity, we see that 
$\bigl(\begin{smallmatrix} 0 & 1 \\ 1 & 0 \end{smallmatrix}\bigr)$
switches ${\cal F}_n^+$ and ${\cal F}_n^-$,
and the description of ${\cal F}_n^-$ follows.
Finally, the case of ${\cal G}_n^-$ is similar to that of ${\cal F}_n^-$.
\end{proof}

Recall that we realize the maximal compact subgroup $U(2) \times U(2)$
of $U(2,2)$ as well as $SU(2) \times SU(2)$ as diagonal subgroups of
$GL(2,\HC)$ using (\ref{U(2)xU(2)})-(\ref{SU(2)xSU(2)}).
%$$
%U(2) \times U(2) = \bigl\{
%\bigl(\begin{smallmatrix} a & 0 \\ 0 & d \end{smallmatrix}\bigr)
%\in GL(2,\HC); \: a,d \in \HC^{\times} \bigr\}.
%$$
%Then $SU(2) \times SU(2)$ is realized as
%$$
%SU(2) \times SU(2) = \bigl\{
%\bigl(\begin{smallmatrix} a & 0 \\ 0 & d \end{smallmatrix}\bigr)
%\in GL(2,\HC); \: a,d \in \BB H,\: N(a)=N(d)=1 \bigr\}.
%$$

\begin{prop}
The spaces ${\cal F}_n^+$, ${\cal F}_n^-$ are invariant under the
$\pi_{nl}$-actions of $U(2) \times U(2)$ and $\mathfrak{gl}(2,\HC)$.
Similarly, the spaces ${\cal G}_n^+$, ${\cal G}_n^-$ are invariant under the
$\pi_{nr}$-actions of $U(2) \times U(2)$ and $\mathfrak{gl}(2,\HC)$.
\end{prop}

\begin{proof}
To show that ${\cal F}_n^+$ is invariant under the $\pi_{nl}$-action of
$\mathfrak{gl}(2,\HC)$, we need to show that if
$F^{(n)}_{l,\mu,\nu}(Z) \in {\cal F}_n^+$ and
$\bigl(\begin{smallmatrix} A & B \\ C & D \end{smallmatrix}\bigr) \in
\mathfrak{gl}(2,\HC)$, then
$\pi_{nl} \bigl(\begin{smallmatrix} A & B \\ C & D \end{smallmatrix}\bigr)
F^{(n)}_{l,\mu,\nu}(Z) \in {\cal F}_n^+$.
Indeed, consider first a special case of $A=B=D=0$, then 
$F^{(n)}_{l,\mu,\nu}(Z)$ is homogeneous of degree $2l$ and
$\pi_{nl} \bigl(\begin{smallmatrix} 0 & 0 \\ C & 0 \end{smallmatrix}\bigr)
F^{(n)}_{l,\mu,\nu}(Z)$ is homogeneous of degree $2l+1$.
By the Cauchy-Fueter formula for $n$-regular functions
(Theorem \ref{Fueter-n-reg}),
\begin{multline}  \label{CF}
\frac{(2l+n+1)!}{(2l+2)!}
\pi_{nl} \bigl(\begin{smallmatrix} 0 & 0 \\ C & 0 \end{smallmatrix}\bigr)
F^{(n)}_{l,\mu,\nu}(W)  \\
= \frac1{2\pi^2} \int_{S^3_R} k_{n/2}(Z-W) \cdot
(Dz \otimes Z \otimes \cdots \otimes Z) \cdot
\Bigl( \pi_{nl} \bigl(\begin{smallmatrix} 0 & 0 \\ C & 0 \end{smallmatrix}\bigr)
F^{(n)}_{l,\mu,\nu}(Z) \Bigr),
\end{multline}
where $S^3_R \subset \BB H$ is the sphere centered at the origin of radius $R$
large enough so that $W \in \BB D^+_R$.
On the other hand, recall the expansion (\ref{k_n-expansion})
of the Cauchy-Fueter kernel $k_{n/2}(Z-W)$.
The integral
$$
\int_{S^3_R} G'^{(n)}_{l',\mu',\nu'}(Z) \cdot
(Dz \otimes Z \otimes \cdots \otimes Z) \cdot
\Bigl( \pi_{nl} \bigl(\begin{smallmatrix} 0 & 0 \\ C & 0 \end{smallmatrix}\bigr)
F^{(n)}_{l,\mu,\nu}(Z) \Bigr) = 0
$$
unless $l'=l+1/2$ (see the proof of Proposition \ref{orthogonality-nreg-prop}
based on the independence of the integral of $R>0$).
Then expansion (\ref{k_n-expansion}) together with (\ref{CF}) show that 
$\pi_{nl} \bigl(\begin{smallmatrix} 0 & 0 \\ C & 0 \end{smallmatrix}\bigr)
F^{(n)}_{l,\mu,\nu}(W)$
is a finite linear combination of $F^{(n)}_{l+1/2,\mu',\nu'}(W)$'s
%with $\mu' \in \BB Z +l+n/2+1/2$, $\nu' \in \BB Z +l+1/2$,
%$-l-n/2-1/2 \le \mu' \le l+n/2+1/2$, $-l-1/2 \le \nu' \le l+1/2$.
and hence an element of ${\cal F}_n^+$.
The other cases can be proved similarly.

As suggested by the referee, another proof of this proposition can be
obtained from Theorem \ref{dr-action-thm} and
Proposition \ref{basis-free-description}.
\end{proof}
  
For $d \in \BB Z$, define
$$
{\cal F}_n^+(d) =
\{ f(Z) \in {\cal F}_n^+;\: \text{$f(Z)$ is homogeneous of degree $d$} \}.
$$
Similarly, we define ${\cal F}_n^-(d)$, ${\cal G}_n^+(d)$ and ${\cal G}_n^-(d)$.
As in our previous papers, we denote by $(\tau_l,V_l)$ the irreducible
$(2l+1)$-dimensional representation of $SU(2)$ or $\mathfrak{sl}(2,\BB C)$,
with $l=0,\frac12,1,\frac32,\dots$.

%We realize $\mathfrak{sl}(2,\BB C) \times \mathfrak{sl}(2,\BB C)$
%as diagonal elements of $\mathfrak{gl}(2,\HC)$:
%$$
%\mathfrak{sl}(2,\BB C) \times \mathfrak{sl}(2,\BB C) = \left\{
%\left(\begin{smallmatrix} A & 0 \\ 0 & D \end{smallmatrix}\right)
%\in \mathfrak{gl}(2,\HC);\: A,D \in \HC, \re(A)=\re(D)=0 \right\}.
%$$

\begin{prop}  \label{K-types-prop}
Each ${\cal F}_n^{\pm}(d)$ is invariant under the $\pi_{nl}$ action restricted
to $SU(2) \times SU(2)$, and we have the following decomposition into
irreducible components:
$$
{\cal F}_n^+(2l) = V_l \boxtimes V_{l+\frac n2}, \qquad
{\cal F}_n^+(d) = \{0\} \quad \text{if $d < 0$},
$$
$$
{\cal F}_n^-(-2l-n-2) = V_{l+\frac n2} \boxtimes V_l, \qquad
{\cal F}_n^-(d) = \{0\} \quad \text{if $d > -(n+2)$},
$$
$l=0,\frac12,1,\frac32,\dots$.

Similarly, each ${\cal G}_n^{\pm}(d)$ is invariant under the $\pi_{nr}$ action
restricted to $SU(2) \times SU(2)$, and we have the following decomposition
into irreducible components:
$$
{\cal G}_n^+(2l) = V_{l+\frac n2} \boxtimes V_l, \qquad
{\cal G}_n^+(d) = \{0\} \quad \text{if $d < 0$},
$$
$$
{\cal G}_n^-(-2l-n-2) = V_l \boxtimes V_{l+\frac n2}, \qquad
{\cal G}_n^-(d) = \{0\} \quad \text{if $d > -(n+2)$},
$$
$l=0,\frac12,1,\frac32,\dots$.
\end{prop}

\begin{rem}
When $n=1$, this result becomes Proposition 21 in \cite{FL1}.
\end{rem}

\begin{proof}
Recall that the actions of $SU(2) \times SU(2)$ on ${\cal F}_n^{\pm}$ and
${\cal G}_n^{\pm}$ are obtained by restricting the actions of $GL(2,\HC)$
described by (\ref{pi_nl})-(\ref{pi_nr}).
It follows that $SU(2) \times SU(2)$ preserves each
${\cal F}_n^{\pm}(d)$ and ${\cal G}_n^{\pm}(d)$.

Next we identify each ${\cal F}_n^{\pm}(d)$ and ${\cal G}_n^{\pm}(d)$
as a representation of $SU(2) \times SU(2)$.
First, we consider the case of ${\cal F}_n^+(d)$.
It is clear from the definition of  ${\cal F}_n^+$
that ${\cal F}_n^+(d) = \{0\}$ when $d<0$. Consider now the case of
$$
{\cal F}_n^+(2l) =  \BB C\text{-span of } \Bigl\{ F^{(n)}_{l,\mu,\nu}(Z);\:
\begin{smallmatrix} \mu \in \BB Z +l+n/2, & & \nu \in \BB Z +l, \\
-l-n/2 \le \mu \le l+n/2, & & -l \le \nu \le l \end{smallmatrix} \Bigr\};
$$
this space has dimension $(2l+1)(2l+n+1)$.
It is clear from (\ref{pi_nl}) that the first copy of $SU(2)$
acts via $\tau_l$ and that, as representations of $SU(2) \times SU(2)$,
$$
{\cal F}_n^+(2l) \simeq V_l \boxtimes V',
$$
where $V'$ is a subrepresentation of $V_{\frac{n-1}2} \otimes V_l$.
Counting dimensions, we find that $\dim V' = 2l+n+1$.
To show that $V' \simeq V_{l+\frac n2}$, consider the action of an element
$d' = \bigl( \begin{smallmatrix} \lambda & 0 \\ 0 & \lambda^{-1}
\end{smallmatrix} \bigr)$ in the second copy of $SU(2)$.
By the representation theory of $SU(2)$, it is sufficient
to show that this element acts on each $F^{(n)}_{l,\mu=-l-\frac n2,\nu}(Z)$ by
multiplication by $\lambda^{2l+n}$.
Indeed, by the definition of $F^{(n)}_{l,\mu,\nu}(Z)$ (equation (\ref{F-def}))
and by Lemma 22 in \cite{FL1}, $F^{(n)}_{l,\mu=-l-\frac n2,\nu}(Z)$ is proportional
to
$$
\underbrace{\begin{pmatrix} 1 \\ 0 \end{pmatrix} \otimes \cdots \otimes
\begin{pmatrix} 1 \\ 0 \end{pmatrix}}_{\text{$n$ times}}
t^l_{\nu\,\underline{-l}}(Z),
$$
and the matrix coefficient $t^l_{\nu\,\underline{-l}}(Z)$ is in turn proportional
to $z_{11}^{l+\nu} z_{21}^{l-\nu}$.
Then it is easy to see from (\ref{pi_nl}) that the element
$h=\bigl( \begin{smallmatrix} 1 & 0 \\ 0 & d' \end{smallmatrix} \bigr)
\in SU(2) \times SU(2)$, where
$d' = \bigl( \begin{smallmatrix} \lambda & 0 \\ 0 & \lambda^{-1}
\end{smallmatrix} \bigr)$, acts on
$$
\underbrace{\begin{pmatrix} 1 \\ 0 \end{pmatrix} \otimes \cdots \otimes
\begin{pmatrix} 1 \\ 0 \end{pmatrix}}_{\text{$n$ times}} z_{11}^{l+\nu} z_{21}^{l-\nu}
$$
by multiplication by $\lambda^{2l+n}$.

To identify ${\cal F}_n^-(-2l-n-2)$ as a representation of $SU(2) \times SU(2)$,
observe that, by Proposition \ref{basis-free-description}, the element
$\bigl( \begin{smallmatrix} 0 & 1 \\ 1 & 0 \end{smallmatrix} \bigr)
\in GL(2,\HC)$ maps ${\cal F}_n^+(2l)$ into ${\cal F}_n^-(-2l-n-2)$
(and vice versa) and switches the two factors in $SU(2) \times SU(2)$.
Finally, we can identify ${\cal G}_n^{\pm}(d)$ using the bilinear pairing
(\ref{pairing1}) which, by Propositions \ref{invariant} and
\ref{orthogonality-nreg-prop}, restricts to non-degenerate
$SU(2) \times SU(2)$-invariant pairings between
$$
{\cal F}_n^+(d) \quad \text{and} \quad {\cal G}_n^-(-d-n-2)
$$
and between
$$
{\cal F}_n^-(d) \quad \text{and} \quad {\cal G}_n^+(-d-n-2).
$$
\end{proof}

We conclude our description of the representations ${\cal F}_n^{\pm}$
and ${\cal G}_n^{\pm}$ by showing their irreducibility.

\begin{thm}  \label{irreducible-thm}
We have:
$$
(\pi_{nl},{\cal F}_n^+), \quad (\pi_{nl},{\cal F}_n^-), \quad
(\pi_{nr},{\cal G}_n^+), \quad (\pi_{nr},{\cal G}_n^-)
$$
are irreducible representations of
$\mathfrak{sl}(2,\HC)$ (as well as $\mathfrak{gl}(2,\HC)$).
\end{thm}

\begin{proof}
We will show that ${\cal F}_n^{\pm}$ are irreducible.
Then the irreducibility of ${\cal G}_n^{\pm}$ is immediate from
Propositions \ref{invariant} and \ref{orthogonality-nreg-prop}.
It is easy to see from the description of the $K$-types of ${\cal F}_n^{\pm}$
given by Proposition \ref{K-types-prop} that it is sufficient to show that
any non-zero vector in ${\cal F}_n^+(d)$ also generates non-zero vectors in
\begin{align*}
&{\cal F}_n^+(d-1) \qquad \text{as long as ${\cal F}_n^+(d-1) \ne \{0\}$
\qquad and} \\
&{\cal F}_n^+(d+1) \qquad \text{as long as ${\cal F}_n^+(d+1) \ne \{0\}$}
\end{align*}
(and similarly for ${\cal F}_n^-(d)$).
This follows from the observation that if $v_d \in {\cal F}_n^{\pm}(d)$, then
$$
\pi_{nl} \left(\begin{smallmatrix} 0 & B \\ 0 & 0 \end{smallmatrix}\right)
v_d \in {\cal F}_n^{\pm}(d-1), \qquad
\pi_{nl} \left(\begin{smallmatrix} 0 & 0 \\ C & 0 \end{smallmatrix}\right)
v_d \in {\cal F}_n^{\pm}(d+1),
$$
for each $\left(\begin{smallmatrix} 0 & B \\ 0 & 0 \end{smallmatrix}\right),
\left(\begin{smallmatrix} 0 & 0 \\ C & 0 \end{smallmatrix}\right)
\in \mathfrak{sl}(2,\HC)$,
Lemma \ref{Lie-alg-action} describing
$\pi_{nl} \left(\begin{smallmatrix} 0 & B \\ 0 & 0 \end{smallmatrix}\right)$
and the fact that conjugation by
$\left(\begin{smallmatrix} 0 & 1 \\ 1 & 0 \end{smallmatrix}\right)
\in GL(2,\HC)$ switches
$\left(\begin{smallmatrix} 0 & B \\ 0 & 0 \end{smallmatrix}\right)$
and $\left(\begin{smallmatrix} 0 & 0 \\ C & 0 \end{smallmatrix}\right)$.
\end{proof}

Note that the results of this section show that ${\cal F}_n^{\pm}$
and ${\cal G}_n^{\pm}$ are irreducible Harish-Chandra modules of a very special
kind -- they have highest or lowest weights.

\section{Unitary Structures on ${\cal F}_n^{\pm}$ and ${\cal G}_n^{\pm}$}  \label{unitary-section}

In this section we describe $\mathfrak{u}(2,2)$-invariant inner products
on $(\pi_{nl},{\cal F}_n^+)$, $(\pi_{nl},{\cal F}_n^-)$, $(\pi_{nr},{\cal G}_n^+)$
and $(\pi_{nr},{\cal G}_n^-)$.
We realize $U(2,2)$ as the subgroup of elements of $GL(2,\HC)$
preserving the Hermitian form on $\BB C^4$ given by the $4 \times 4$ matrix
$\bigl(\begin{smallmatrix} 1 & 0 \\ 0 & -1 \end{smallmatrix}\bigr)$.
Explicitly,
\begin{equation*}
U(2,2) = \Biggl\{ \begin{pmatrix} a & b \\ c & d \end{pmatrix};\:
a,b,c,d \in \HC,\:
\begin{matrix} a^*a = 1+c^*c \\ d^*d = 1+b^*b \\ a^*b=c^*d \end{matrix}
\Biggr\}.
%&= \Biggl\{ \begin{pmatrix} a & b \\ c & d \end{pmatrix};\:
%a,b,c,d \in \HC,\:
%\begin{matrix} a^*a = 1+b^*b \\ d^*d = 1+c^*c \\ ac^*=bd^* \end{matrix}
%\Biggr\}.
\end{equation*}
(Recall that $a^*$ and $d^*$ denote the matrix adjoints of $a$ and $d$ in the
standard realization of $\HC$ as $2 \times 2$ complex matrices.)
Then the  Lie algebra of $U(2,2)$ is
\begin{equation}  \label{u(2,2)-algebra}
\mathfrak{u}(2,2) = \Bigl\{
\begin{pmatrix} A & B \\ B^* & D \end{pmatrix} ;\: A,B,D \in \HC ,\:
A=-A^*, D=-D^* \Bigr\}.
\end{equation}

Recall from Subsection 2.3 of \cite{FL1} the $\BB C$-antilinear maps
$\BB S \to \BB S'$ and $\BB S' \to \BB S$
-- matrix transposition followed by complex conjugation.
Since they are similar to quaternionic conjugation,
we use the same symbol for these maps:
$$
\begin{pmatrix} s_1 \\ s_2 \end{pmatrix}^+ = (\overline{s_1}, \overline{s_2}),
\qquad
(s'_1,s'_2)^+ =\begin{pmatrix} \overline{s'_1} \\ \overline{s'_2}\end{pmatrix},
\qquad
\begin{pmatrix} s_1 \\ s_2 \end{pmatrix} \in \BB S ,\: (s'_1,s'_2) \in \BB S'.
$$
Note that
\begin{equation}  \label{spinor-conj}
(Zs)^+ = s^+ \bar Z^+, \qquad (s'Z)^+ = \bar Z^+ s'^+, \qquad
Z \in \HC,\: s \in \BB S,\: s' \in \BB S',
\end{equation}
where $\bar Z$ denotes complex conjugation relative to $\BB H$:
\begin{equation}  \label{complex-conj}
\overline{\begin{pmatrix} z_{11} & z_{12} \\ z_{21} & z_{22} \end{pmatrix}}
= \begin{pmatrix} \overline{z_{22}} & - \overline{z_{21}} \\
- \overline{z_{12}} & \overline{z_{11}} \end{pmatrix}.
\end{equation}
These conjugation maps extend to tensor products
$$
\underbrace{\BB S \otimes \cdots \otimes \BB S}_{\text{$n$ times}}
\qquad \text{and} \qquad
\underbrace{\BB S' \otimes \cdots \otimes \BB S'}_{\text{$n$ times}}
$$
in the most obvious way.
We introduce a map $\sigma: \hat {\cal F}_n \to \hat {\cal G}_n$ defined by
$$
\sigma(f)(Z) = f^+(\bar Z^+) = f^+(Z^*), \qquad \text{then} \qquad
\sigma^{-1}(g)(Z) = g^+(\bar Z^+) = g^+(Z^*).
$$
%By Lemma 7 from \cite{FL2},
Then $\sigma$ produces an isomorphism of {\em real} vector spaces
${\cal R}_n \to {\cal R}'_n$, we call its inverse $\sigma^{-1}$.
Clearly, $\sigma$ and $\sigma^{-1}$ map functions that are polynomial on
$\HC$ (respectively $\HC^{\times}$) into functions that are polynomial on
$\HC$ (respectively $\HC^{\times}$).
By Proposition \ref{basis-free-description}, $\sigma$ and $\sigma^{-1}$
restrict to isomorphisms of real vector spaces
$$
\sigma: {\cal F}_n^+ \to {\cal G}_n^+, \qquad
\sigma^{-1}: {\cal G}_n^+ \to {\cal F}_n^+, \qquad
\sigma: {\cal F}_n^- \to {\cal G}_n^-, \qquad
\sigma^{-1}: {\cal G}_n^- \to {\cal F}_n^-.
$$
From (\ref{spinor-conj})-(\ref{complex-conj}) and Lemma \ref{Lie-alg-action}
we obtain the following commutation relations between $\sigma$ and the Lie
algebra actions $\pi_{nl}$, $\pi_{nr}$ of $\mathfrak{gl}(2,\HC)$:
\begin{align}
\sigma \circ \pi_{nl} \bigl( \begin{smallmatrix} 0 & B \\ 0 & 0
\end{smallmatrix} \bigr) f
%&= \pi_{nr} \bigl( \begin{smallmatrix} 0 & \bar B^+ \\ 0 & 0 \end{smallmatrix}
%\bigr) \circ \sigma(f)
&= \pi_{nr} \bigl( \begin{smallmatrix} 0 & B^* \\ 0 & 0 \end{smallmatrix}
\bigr) \circ \sigma(f),  \label{sigma-B-comm}  \\
\sigma \circ \pi_{nl} \bigl( \begin{smallmatrix} 0 & 0 \\ C & 0
\end{smallmatrix} \bigr) f
%&= \pi_{nr} \bigl( \begin{smallmatrix} 0 & 0 \\ \bar C^+ & 0 \end{smallmatrix}
%\bigr) \circ \sigma(f)
&= \pi_{nr} \bigl( \begin{smallmatrix} 0 & 0 \\ C^* & 0 \end{smallmatrix}
\bigr) \circ \sigma(f),
\qquad B, C \in \HC.  \label{sigma-C-comm}
\end{align}

Let us consider pairings on ${\cal F}_n^{\pm}$ and ${\cal G}_n^{\pm}$:
\begin{align}
(f_1,f_2)_{{\cal R}_n} &= \frac1{2\pi^2}
\int_{Z \in S^3} f_2^+(Z) \cdot (\ddd^{-1} f_1)(Z) \,dS,
\qquad f_1, f_2 \in {\cal F}_n^{\pm},  \label{F-inn-prod}  \\
(g_1,g_2)_{{\cal R}'_n} &= \frac1{2\pi^2}
\int_{Z \in S^3} (\ddd^{-1} g_1)(Z) \cdot g_2^+(Z)\,dS,
\qquad g_1, g_2 \in {\cal G}_n^{\pm},  \label{G-inn-prod}
\end{align}
where $S^3 \subset \BB H \subset \HC$ is the unit sphere centered at the origin.
Clearly, these pairings are complex anti-linear, positive definite on
${\cal F}_n^+$ and ${\cal G}_n^+$, and either positive or negative definite
on ${\cal F}_n^-$ and ${\cal G}_n^-$, depending on the effect of the operator
$\ddd^{-1}$.

\begin{thm}  \label{unitary-thm}
The pairing (\ref{F-inn-prod}) is $\mathfrak{u}(2,2)$-invariant on
$(\pi_{nl}, {\cal F}_n^+)$ and $(\pi_{nl}, {\cal F}_n^-)$.
Similarly, the pairing (\ref{G-inn-prod}) is $\mathfrak{u}(2,2)$-invariant on
$(\pi_{nr}, {\cal G}_n^+)$ and $(\pi_{nr}, {\cal G}_n^-)$.
(The Lie algebra $\mathfrak{u}(2,2)$ is realized as a real form of
$\mathfrak{gl}(2,\HC)$ as in (\ref{u(2,2)-algebra}).)
\end{thm}

\begin{proof}
We will prove the invariance of (\ref{F-inn-prod}) only,
the other case is similar.
First, we relate (\ref{F-inn-prod}) to the bilinear pairing (\ref{pairing1}):
$$
(f_1,f_2)_{{\cal R}_n} = \bigl\langle f_1, \sigma \circ
\pi_{nl}\bigl(\begin{smallmatrix} 0 & 1 \\ 1 & 0 \end{smallmatrix}\bigr) f_2
\bigr\rangle_{{\cal R}_n}.
$$
Using the $\mathfrak{gl}(2,\HC)$-invariance of the bilinear pairing
(\ref{pairing1}) (Proposition \ref{invariant}) and relations
(\ref{sigma-B-comm})-(\ref{sigma-C-comm}), we find:
\begin{multline*}
\bigl( \pi_{nl} \bigl( \begin{smallmatrix} 0 & B \\ 0 & 0 \end{smallmatrix}
\bigr) f_1, f_2 \bigr)_{{\cal R}_n}
= \bigl\langle \pi_{nl} \bigl( \begin{smallmatrix} 0 & B \\ 0 & 0
\end{smallmatrix} \bigr)f_1, \sigma \circ
\pi_{nl}\bigl(\begin{smallmatrix} 0 & 1 \\ 1 & 0 \end{smallmatrix}\bigr) f_2
\bigr\rangle_{{\cal R}_n}  \\
= - \bigl\langle f_1, \pi_{nr} \bigl( \begin{smallmatrix} 0 & B \\ 0 & 0
\end{smallmatrix} \bigr) \circ \sigma \circ
\pi_{nl}\bigl(\begin{smallmatrix} 0 & 1 \\ 1 & 0 \end{smallmatrix}\bigr) f_2
\bigr\rangle_{{\cal R}_n}  \\
= - \bigl\langle f_1, \sigma \circ
\pi_{nl}\bigl(\begin{smallmatrix} 0 & 1 \\ 1 & 0 \end{smallmatrix} \bigr) \circ
\pi_{nl} \bigl( \begin{smallmatrix} 0 & 0 \\ B^* & 0 \end{smallmatrix} \bigr) f_2
\bigr\rangle_{{\cal R}_n}
= - \bigl( f_1, \pi_{nl} \bigl( \begin{smallmatrix} 0 & 0 \\ B^* & 0
\end{smallmatrix}\bigr) f_2 \bigr)_{{\cal R}_n}
\end{multline*}
and, similarly,
$$
\bigl( \pi_{nl} \bigl( \begin{smallmatrix} 0 & 0 \\ C & 0 \end{smallmatrix}
\bigr) f_1, f_2 \bigr)_{{\cal R}_n}
= - \bigl( f_1, \pi_{nl} \bigl( \begin{smallmatrix} 0 & C^* \\ 0 & 0
\end{smallmatrix}\bigr) f_2 \bigr)_{{\cal R}_n}.
$$
This proves the invariance of (\ref{F-inn-prod}) under
$\bigl(\begin{smallmatrix} 0 & B \\ B^* & 0 \end{smallmatrix}\bigr)
\in \mathfrak{u}(2,2)$, $B \in \HC$.
To show that the pairing (\ref{F-inn-prod}) is invariant under
$\bigl(\begin{smallmatrix} A & 0 \\ 0 & D \end{smallmatrix}\bigr)
\in \mathfrak{u}(2,2)$, $A,D \in \HC$, $A=-A^*$, $D=-D^*$,
one checks directly that (\ref{F-inn-prod}) is invariant under
the group $U(2) \times U(2)$ realized as (\ref{U(2)xU(2)}).
\end{proof}

\noindent
{\em Department of Mathematics, Yale University,
P.O. Box 208283, New Haven, CT 06520-8283}\\
{\em Department of Mathematics, Indiana University,
Rawles Hall, 831 East 3rd St, Bloomington, IN 47405}   

\end{document}